\documentclass[english, 11pt]{amsart}
\usepackage{amssymb}
\usepackage{amsfonts}
\usepackage{amscd}
\usepackage[T1]{fontenc}
\usepackage{color}
\newbox\removebox
\newcommand\remove[1]{%
\setbox\removebox=\ifmmode\hbox{$#1$}\else\hbox{#1}\fi%
\leavevmode
\rlap{\textcolor{blue}{\vrule height0.8ex depth-0.6ex width\wd\removebox}}%
\box\removebox
}
\long\def\bigremove#1{%
\par\setbox\removebox=\vbox{#1}%
\vbox{%
\vbox to0pt{\hbox{\tikz\draw[color=blue,thick] (0,0) -- (\wd\removebox,-\ht\removebox)  (\wd\removebox,0) -- (0,-\ht\removebox);}}
\box\removebox
}
}

\usepackage{mathrsfs} %for the \mathscr

\newcommand{\cCexp}{\cC^{\mathrm{exp}}}
\newcommand{\cQexp}{\cQ^{\mathrm{exp}}}

\newcommand{\AO}[1][]{\cA_{\cO#1}}
\newcommand{\BO}[1][]{\cB_{\cO#1}}
\newcommand{\CO}[1][]{\mathcal{C}_{\cO#1}}
\newcommand{\Vol}{\operatorname{Vol}}
\newcommand{\Int}{\operatorname{Int}}
\newcommand{\Bdd}{\operatorname{Bdd}}
\newcommand{\Iva}{\operatorname{Iva}}

\def\RDefe{Q^{\rm exp}}

\def\ac{{\overline{\rm ac}}}

\def\Def{\operatorname{Def}}

\def\RDef{\operatorname{RDef}}

\def\LPas{\cL_{\rm DP}}

\def\res{\operatorname{res}}

\def\11{{\mathbf 1}}
\def\AA{{\mathbb A}}

\def\CC{{\mathbb C}}

\def\FF{{\mathbb F}}

\def\LL{{\mathbb L}}

\def\NN{{\mathbb N}}

\def\QQ{{\mathbb Q}}
\def\RR{{\mathbb R}}

\def\ZZ{{\mathbb Z}}

\def\cA{{\mathcal A}}
\def\cB{{\mathcal B}}
\def\cC{{\mathscr C}}
\def\cD{{\mathcal D}}

\def\cH{{\mathcal H}}

\def\cL{{\mathcal L}}
\def\cM{{\mathcal M}}

\def\cO{{\mathcal O}}
\def\cP{{\mathcal P}}
\def\cQ{{\mathcal Q}}

\def\llp{\mathopen{(\!(}}

\def\rrp{\mathopen{)\!)}}

\newtheorem{thm}[subsubsection]{Theorem}
\newtheorem{lem}[subsubsection]{Lemma}
\newtheorem{cor}[subsubsection]{Corollary}
\newtheorem{prop}[subsubsection]{Proposition}

\theoremstyle{definition}
\newtheorem{defn}[subsubsection]{Definition}

\newtheorem{def-prop}[subsubsection]{Proposition-Definition}
\newtheorem{def-theorem}[subsubsection]{Theorem-Definition}
\newtheorem{def-lem}[subsubsection]{Lemma-Definition}

\theoremstyle{remark}
\newtheorem{remark}[subsubsection]{Remark}

\theoremstyle{plain}

\numberwithin{equation}{subsection}

  {\par\medskip\noindent #1\par\begingroup%
    \advance\leftskip by 1em\advance\rightskip by 1em}%
  {\par\endgroup}

\newcommand{\sq}{\mathrel{\square}}
\newcommand{\ord}{\operatorname{ord}}

\newcommand{\statement}[1]{\\\hspace*{\fill}#1\hspace*{\fill}\\}

\begin{document}

\setcounter{tocdepth}{1} % Show subsection in table of contents

\author{Raf Cluckers}
\address{Universit\'e Lille 1, Laboratoire Painlev\'e, CNRS - UMR 8524, Cit\'e Scientifique, 59655
Villeneuve d'Ascq Cedex, France, and,
KU Leuven, Department of Mathematics,
Celestijnenlaan 200B, B-3001 Leu\-ven, Bel\-gium}
\email{Raf.Cluckers@math.univ-lille1.fr}
\urladdr{http://math.univ-lille1.fr/$\sim$cluckers}

\author{Julia Gordon}
\address{Department of Mathematics, University of British Columbia,
Vancouver BC V6T 1Z2 Canada}
\email{gor@math.ubc.ca}
\urladdr{http://www.math.ubc.ca/$\sim$gor}

\author{Immanuel Halupczok}
\address{Institut f\"ur Mathematische Logik und Grundlagenforschung
Universit\"at M\"unster\\
Einsteinstra\ss{}e 62\\
48149 M\"unster\\
Germany}
\email{math@karimmi.de}
\urladdr{http://www.immi.karimmi.de/en.math.html}



\title[Transfer principles for integrability and boundedness]
{Integrability of
oscillatory functions on local fields: transfer principles}

\begin{abstract}
For oscillatory functions on local fields coming from motivic exponential functions, we show that integrability over ${\mathbb Q}_p^n$ implies integrability over ${\mathbb F}_p \mathopen{(\!(} t \mathopen{)\!)} ^n$ for large $p$, and vice versa. More generally, the integrability only depends on the isomorphism class of the residue field of the local field, once the characteristic of the residue field is large enough.
This principle yields general local integrability results for Harish-Chandra characters in positive characteristic as we show in other work.
Transfer principles for related conditions such as boundedness and local integrability are also obtained.
The proofs %of the transfer principles
rely on a thorough study of loci of integrability, to which we give a geometric meaning %to these analytically defined loci
by relating them to zero loci of functions of a specific kind.
%We give a geometric meaning to these analytically defined loci by relating them to specific zero loci.
\end{abstract}

%\thanks{The research leading to these results has received funding from the European Research Council under the European Community's Seventh Framework Programme (FP7/2007-2013) / ERC Grant Agreement nr. 246903 NMNAG}

\maketitle

%\tableofcontents

\section{Introduction}

Integrability conditions of oscillatory functions are often important but difficult
to control. The idea to relate integrability conditions over $\FF_p\llp t \rrp^n$ to integrability over $\QQ_p^n$ may sound tricky, the more so since integrability is not a first order property (in any seemingly natural language), and thus plays at a different level than, say, the classical Ax-Kochen principle \cite{AK2}. In order to relate integrability conditions over different local fields one must first relate oscillatory functions between the corresponding fields. To this end, we use the motivic exponential functions of \cite{CLexp}, which specialize naturally to oscillatory functions on non-archimedean local fields.
Such motivic exponential functions play an increasingly important role in representation theory, for example, they played a crucial role in one of the methods for obtaining the Fundamental Lemma of the Langlands program and many of its variants in the characteristic zero case, see  \cite{CHL} and the appendix of \cite{YGordon}.

% (in \cite{walds} an alternative approach is followed).

%In this paper we push the understanding of integrability of these functions much further.

%As a famous example, the integrals featuring in the Fundamental Lemma of the Langlands program and many of its variants are known to come from motivic exponential functions, which has important implications, see  \cite{CHL} and the appendix of \cite{YGordon}.

The main result of this paper is the transfer principle that relates $L^1$-integrability over $\FF_p\llp t \rrp^n$ to $L^1$-integrability over $\QQ_p^n$ for specializations of motivic exponential functions, and similarly for any pair of local fields with isomorphic residue fields.
The first application is the proof of the local integrability of Harish-Chandra characters in positive (large) characteristic, that we give in \cite{CGH2}.

The proofs involve a precise understanding
of what we call ``loci of integrability'':
for a family of functions $f_x$ with parameter $x$, the locus of integrability is the set of those parameters $x$ for which the function $f_x$ is integrable (see Definition \ref{loci}).
Integrability being a subtle and analytic condition, it is surprising that we can relate it back to geometry, that is, we relate loci of integrability to zero loci of functions of specific kinds. This geometric viewpoint on loci of integrability is key to proving our new transfer principles.

%together with results from \cite{CLexp} are key to prove our new transfer principles.

Let us explain basic versions of the transfer principles for a simple kind of motivic exponential functions.
Consider a formal expression $F$ of the form
$$
\sum_{i}|f_i|\cdot E(h_i), % \prod_j \ord (g_{ij}).
$$
where $E$ is a formal symbol, and where the $f_i$  and $h_i$ are (first order) definable functions in the language of valued fields with the same domain $X$ and taking values in the valued field.
For any non-archimedean local field $K$ %of large enough residue field characteristic
and for any non-trivial additive character $\psi:K\to\CC^\times$ the expression $F$ specializes to the oscillatory function
$$
F_{K,\psi}:X_K \to\CC: x\mapsto \sum_{i}|f_{iK}(x)|\psi(h_{iK}(x)), %\prod_j \ord (g_{ijK}(x)),
$$
where $|\cdot|$ is the norm on $K$, $X_K$ is the subset of $K^n$ for some $n$ obtained by interpreting the formula defining $X$ in $K$,
and where the $f_{iK}:X_K\subset K^n\to K$ are the interpretations in $K$ of the $f_i$ and similarly for the $h_i$. There are many ways to endow $X_K$ with a measure, a basic example being the restriction of the Haar measure on $K^n$ to $X_K$ in the case when $X_K$ is open in $K^n$.
We can now state a basic form of one of our main results.
\begin{quote}\textbf{Transfer principle for integrability (basic form). }
As soon as the residue field $k_K$ of $K$ has sufficiently large characteristic, whether the statement
\statement{
$F_{K,\psi}$ is $L^1$-integrable over $X_K$ for each $\psi$
}
holds or not, depends only on (the isomorphism class of) $k_K$.
\end{quote}

For applications one typically needs the family version of this transfer principle as given by Theorem \ref{mtrel}.
Variants with conditions like boundedness, local integrability, and local boundedness are also obtained, see Theorem \ref{mtbound}.
%We also prove similar transfer principles for other conditions like local integrability and (local) boundedness, see Theorem \ref{mtbound}.
Thanks to the ubiquity of motivic exponential functions, the results apply to a large class of functions arising in representation theory, such as the
functions representing the Fourier transforms of orbital integrals. In particular our results lead to the first general proof of the local integrability of Harish-Chandra characters in large characteristics \cite{CGH2}, generalizing the results of e.g.~\cite{lemaire:gln}.
%A possible alternative method to some results of \cite{CGH2} may be based on more classical techniques from e.g.~\cite{walds}.
While the work on the Fundamental Lemma in characteristic zero of \cite{CHL} and the appendix of \cite{YGordon} combine insights about motivic exponential functions with results in positive characteristics of \cite{Ngo} and \cite{YGordon}, the results of \cite{CGH2} combine the theorems of Harish-Chandra \cite{Harish} in characteristic zero with Theorem \ref{mtrel} of this paper.
Waldspurger \cite{walds} has given an alternative approach for some results of \cite{CHL}, but it remains to be shown that techniques of e.g.~\cite{walds} may also recover some results on integrability of \cite{CGH2}.

Let us begin with the general definitions of loci of integrability, of boundedness, and of identical vanishing.

For arbitrary sets $A\subset X\times T$ and $x\in X$, write $A_x$ for the set of $t\in T$ with $(x,t)\in A$.
For $g:A\subset X\times T\to B$ a function
and for $x\in X$, write $g(x,\cdot)$ for the function $A_x\to B$ sending $t$  to $g(x,t)$.

Let $T$ and $X$ be arbitrary sets, and let $f:X\times T\to \CC$ be a function.
\begin{defn}\label{loci}
Define the \emph{locus of boundedness of $f$ in $X$} as the set
$$
\Bdd(f,X) := \{x\in X\mid f(x,\cdot) \mbox{ is bounded on }T\}.
$$
Define the \emph{locus of identical vanishing of $f$ in $X$} as the set
$$
\Iva(f,X) := \{x\in X\mid f(x,\cdot) \mbox{ is identically zero on }T\}.
$$
If moreover $T$ is equipped with a complete measure, we define the \emph{locus of integrability of $f$ in $X$} as the set
$$
\Int(f,X) := \{x\in X\mid f(x,\cdot) \mbox{ is measurable and integrable over }T\}.
$$
\end{defn}

We prove that loci of integrability for functions $f_K$ coming from a motivic exponential function $f$ on $X\times K^n$ are actually zero loci of functions coming from a motivic exponential function on $X$, and similarly for loci of boundedness and identical vanishing, see Theorem \ref{p2pexpmot}.

In fact, we develop all our results gradually: first in the case of summation over the integers in Section \ref{s:pres}, then for $p$-adic integration in Section \ref{sec:qfixp}, and finally for motivic integration in Section \ref{s:mot}.
We also sharpen the results of \cite{CLoes} and \cite{CLexp} about stability under integration, see e.g.~Theorems \ref{p1pexp} and \ref{p1qmotexp}, and we give general interpolation results of given functions by integrable functions, see e.g.~Theorems \ref{p1qmot} and \ref{p2'mot}. It is the precise study of these loci that allows us to prove our new transfer principles.
The biggest challenges that we had to address were of course related to oscillation, which appears from Section \ref{osc} on.
An important ingredient to control oscillation is the technical
Proposition \ref{intCLpexp}, which, in a certain sense, states that oscillation
can not interact too badly with definable conditions.

\subsection{Conventions}\label{sec:loci}

For a function $f:A\to \CC$, we write $Z(f)$ for the zero locus $\{a\in A\mid f(a)=0\}$ of $f$, and similarly for an $R$-valued function $f:A\to R$ for any ring $R$.

For sets $A_1, A_2, X$ and functions $f_i:A_i \to X$, a function $g:A_1\to A_2$ is said to be over $X$ when $f_2\circ g= f_1$.
%$g$ makes a commutative diagram with the maps $f_i$.
Often, the $f_i$ will be coordinate projections to $X$.

Definition \ref{loci} will typically be applied to the counting measure on $\ZZ$,
to the Haar measure on $\QQ_p$ normalized such that $\ZZ_p$ has
measure $1$, and to product measures on Cartesian products of these sets.
Our results also apply mutatis mutandis to measures on varieties $V$ over $\cO_K$ by working with affine charts and volume forms, see \cite{CLoes}, \cite{CLbounded}.

Recall that a complex valued functions is called bounded if and only if its range is contained in a compact subset of $\CC$.

\subsection*{Acknowledgments}
\hspace{0.5cm}
The authors are deeply indebted to T.~Hales, without whose influence the present paper would not have been written in the present form.
We would also like to thank F.~Loeser and J.~Denef for their deep insights in the subject and interesting discussions during the preparation of this paper, and the referee for several useful comments. The authors were supported in part by the European Research Council under the European Community's Seventh Framework Programme (FP7/2007-2013) with ERC Grant Agreement nr. 246903 NMNAG,
by the Labex CEMPI  (ANR-11-LABX-0007-01), by the the Fund for Scientific Research of Flanders, Belgium (grant G.0415.10),
by the SFB~878 of the Deutsche Forschungsgemeinschaft, and by NSERC.
%The third author was supported by the SFB~878 of the Deutsche Forschungsgemeinschaft, and the second author was supported by NSERC.

\section{Summability over the integers}\label{s:pres}

Summation over the integers and the results presented in this section are important for us since they lie behind $p$-adic integration: several of the $p$-adic results of Section \ref{sec:qfixp}, and even some of the motivic results of Section \ref{s:mot}, will be reduced to the results of this section.

\subsection{Presburger with base $q$}\label{sec:qfix}
In this section, let $q>1$ be a fixed real number.

By a Presburger set, one means a subset of $\ZZ^m$ for some $m\geq 0$ which can be described by a Boolean combination of sets of the following forms
$$
\{x\in\ZZ^m\mid f(x)\geq 0\}
$$
$$
\{x\in\ZZ^m\mid g(x)\equiv 0\bmod n\},
$$
where $f$ and $g$ are polynomials over $\ZZ$ of degree $\leq 1$, and $n>0$ is an integer.
A Presburger function is a function between Presburger sets whose graph is also a Presburger set.
A Presburger function is called linear if it is the restriction of an affine map $\QQ^k\to\QQ^\ell$.
We write $\NN$ for the set of non-negative integers $\{z\in\ZZ\mid z\geq 0\}$.

\begin{defn}
Define the subring $\AA_q\subset\RR$ as
$$
\AA_q =  \ZZ\left[q,q^{-1},\left(\frac{1}{1-q^{-i}}\right)_{i\in \NN,\ 0<i}\right].
$$
For $S$ a Presburger set, let $\cP_q(S)$ be the $\AA_q$-algebra of $\AA_q$-valued functions on $S$ generated by
 \begin{enumerate}
\item all Presburger functions $\alpha:S\to\ZZ$,

\item the functions $q^\beta:S\to \AA_q:s\mapsto q^{\beta(s)}$ for all Presburger functions $\beta:S\to\ZZ$.
\end{enumerate}
\end{defn}

The functions in $\cP_q(S)$ are called Presburger constructible functions on $S$ (with base $q$).
Note that a general function in $\cP_q(S)$ is of the form
$$
s\mapsto \sum_{i=1}^N a_i q^{\beta_i(s)}\prod_{j=1}^{M_i}\alpha_{ij}(s),
$$
with the $\alpha_{ij}$ and $\beta_i$ Presburger functions $S\to\ZZ$, and the $a_i$ elements of $\AA_q$.
The constants $\frac1{1-q^{-i}}$ are needed in $\AA_q$ to make the
framework closed under summation; see Theorem~\ref{p1q-new}.

By the quantifier elimination results of \cite{Presburger}, the image of a Presburger set under a Presburger function is again a Presburger set, as are finite intersections, finite unions, and complements of Presburger sets. The situation for zero loci of Presburger \emph{constructible} functions is much more delicate. For finite unions and finite intersections there is no difficulty, as follows.
\begin{lem}\label{inters}
Let $S$ be a Presburger set and let $h_i$ be in $\cP_q(S)$ for $i=1,\ldots,N$.
Consider zero loci
$$
Z(h_i) = \{s\in S\mid h_i(s)=0\}.
$$
Then there exist $f$ and $g$ in $\cP_q(S)$ such that
$$
Z(f) = \bigcap_{i=1}^{N} Z(h_i)\ \mbox{ and }\ Z(g)= \bigcup_{i=1}^{N} Z(h_i).
$$
\end{lem}
\begin{proof}
One can just take the sum of the squares, resp.~the product, of the $h_i$.
\end{proof}
Note that any Presburger subset $A$ of $S$ appears as the zero locus of a Presburger constructible function on $S$;
indeed, $A=Z(1-\chi_A)$ with $\chi_A$ the characteristic function of $A$ in $S$. However, the converse is not true: there are Presburger constructible functions whose zero locus is not a Presburger set. Moreover, for $h\in \cP_q(S)$, the complement of $Z(h)$
in $S$ is not always equal to the zero locus of some function in $\cP_q(S)$.
It turns out that zero loci of Presburger constructible functions are closely related to loci of integrability, of boundedness, and of identical vanishing.
Indeed, the zero loci of Presburger constructible functions are exactly the sets that arise as loci of integrability (against the counting measure) of Presburger constructible functions, and similarly for the loci of boundedness and of identical vanishing.

\begin{thm}[Correspondences of loci]
\label{p2}
Let $f$ be in $\cP_q(S\times \ZZ^m)$ for some Presburger set $S$ and some $m\geq 0$.
Then there exist $h_1,h_2$ and $h_3$ in $\cP_q(S)$ such that
\begin{equation}\label{p2eq}
\Int(f,S) = Z(h_1), %\{s\in S\mid h_1(s)=0\},
\end{equation}
\begin{equation}\label{p2b}
\Bdd(f,S) = Z(h_2), %\{s\in S\mid h_2(s)=0\},
\end{equation}
and
\begin{equation}\label{p3}
\Iva(f,S) = Z(h_3), % \{s\in S\mid h_3(s)=0\},
\end{equation}
where integrability in (\ref{p2eq}) is with respect to the counting measure on $\ZZ^m$, and where $Z(\cdot)$ stands for the zero locus.
\end{thm}
\begin{remark}\label{remvice}
Theorem~\ref{p2} implies that the classes of sets which can appear as different kinds of loci for Presburger constructible functions are all equal,
 since for any given function $h$ in $\cP_q(S)$, there exists $f\in \cP_q(S\times \ZZ)$ such that
$$
Z(h) = \Int(f,S) = \Bdd(f,S) = \Iva(f,S).
$$
Indeed, one can take $f(s,y)=h(s)\cdot y$. Hence the name `correspondences of loci' for Theorem \ref{p2}.
\end{remark}

One can interpolate Presburger constructible functions by Presburger constructible functions with maximal locus of integrability, as follows.

\begin{thm}[Interpolation]\label{p2'}
Let $f$ be in $\cP_q(S\times \ZZ^m)$ for some Presburger set $S$ and some $m\geq 0$.
Then there exists $g$ in $\cP_q(S\times \ZZ^m)$ with $\Int(g,S)=S$ and such that
$f(s,y)=g(s,y)$ whenever $s$ lies in $\Int(f,S)$.
\end{thm}

The following result on stability of $\cP_q$ under summation generalizes Theorem-Definition 4.5.1 of \cite{CLoes} which in turn goes back to Lemma 3.2 of \cite{Denef3}. Theorem-Definition 4.5.1 of \cite{CLoes} is the special case of Theorem \ref{p1q-new} for which $\Int(f,S)=S$.

\begin{thm}
[Integration]\label{p1q-new}
Let $f$ be in $\cP_q(S\times \ZZ^m)$ for some Presburger set $S$ and some $m\geq 0$. Then
there exists a function $g \in \cP_q(S)$ such that
$$
g(s) = \sum_{y\in\ZZ^m} f(s,y)
$$
whenever $s \in \Int(f,S)$.
\end{thm}

Before proving Theorems \ref{p2}, \ref{p2'} and \ref{p1q-new}, we give some auxiliary results.
The first auxiliary lemma can be obtained as a direct corollary of Wilkie's Theorem of \cite{WilkieRexp} on the o-minimality of the real number field with the exponential function.
\begin{lem}\label{lemqexp0}
Let $h:\RR_{\geq 0}\to \RR$ be a function of the form
$$
h(x)=\sum_{i=1}^r c_i x^{a_i} b_{i}^x,
$$
where the $a_i,b_i,c_i$ are real numbers, the $c_i$ and $b_i$ are nonzero, and $r\geq 1$. Suppose that the pairs $(a_i,b_i)$ are mutually different for different $i$. Then the number of zeros of $f$ is bounded by a constant only depending on $r$.
\end{lem}
\begin{proof}
All functions $h$ of the above form but with fixed $r$ are members of a single definable family of functions with discrete zeros in the o-minimal structure of the real number field enriched with the exponential function. Now just note that discrete sets which appear as members of a family of sets in an o-minimal structure are finite and uniformly bounded in size, cf.~\cite{vdD}.
\end{proof}

\begin{lem}\label{lemqcon0}
Let $h:\NN^m\to \RR$ be a function of the form
$$
h(x)=\sum_{i=1}^r c_i q^{b_{i1}x_1+\cdots b_{im}x_m} \prod_{j=1}^m x_j^{a_{ij}},
$$
where the $c_i$ are nonzero real numbers, the $a_{ij}$ and $b_{ij}$ are integers, $a_{ij}\geq 0$, $m\geq 1$, and $r\geq 1$. Suppose that the tuples $(a_{i1},\ldots,a_{im},b_{i1},\ldots,b_{im})$ are mutually different for different $i$. Then $h$ is not identically zero. Furthermore, $h$ is summable over $\NN^m$ if and only if $b_{ij}\leq -1$ for all $i,j$. Finally, $h$ is bounded if and only if simultaneously all $b_{ij}$ are $\leq 0$ and for each $i,j$ with $b_{ij}=0$ one has $a_{ij}=0$.
\end{lem}
\begin{proof}
That $h$ is not identically zero easily follows by induction on $m$ and by Lemma \ref{lemqexp0} for the case $m=1$.
If all the $b_{ij}$ are $<0$ then clearly $h$ is summable. For the other direction,
suppose that $h$ is summable but some $b_{ij}$ is $\geq 0$, say, $b_{11}\geq 0$. We may suppose that $b_{11}$ is maximal among the $b_{ij}$. Put $I=\{i\mid b_{i1}=b_{11}\}$. We may suppose that $a_{11}$ is maximal among the $a_{i1}$ with $i\in I$. Put $J=\{i\in I\mid a_{i1}=a_{11}\}$.
Then the function
$$
x \mapsto q^{b_{11}x_1}x_1^{a_{11}}\sum_{i\in J}  c_i q^{ b_{i2}x_2+\cdots b_{im}x_m }\prod_{j=2}^m x_j^{a_{ij}}
$$
must be identically zero on $\NN^m$ by the summability of $h$, which is impossible by the first statement of the lemma. The statement about boundedness is obtained similarly.
\end{proof}

The following result, Theorem 3 of \cite{CPres}, will be the basis for the proofs of the results in this section.
\begin{thm}[Parametric Rectilinearization \cite{CPres}]\label{paramrecti}
Let $S$ and $X\subset S\times \ZZ^m$ be Presburger sets. Then there exists a finite partition of $X$ into
Presburger sets such that for each part $A$, there is a Presburger set $B\subset S\times \ZZ^{m}$
and a linear Presburger bijection $\rho:A\to B$ over $S$
such that, for each $s\in S$, the set
$B_s$
is
a set of the form
$\Lambda_s\times \NN^\ell$ for a finite subset $\Lambda_s\subset \NN^{m-\ell}$ depending on $s$ and for an integer $\ell\geq 0$ only depending on
$A$.
\end{thm}
Recall that $B_s$ in Theorem \ref{paramrecti} is the set $\{z\in \ZZ^m\mid (s,z)\in B\}$, and that for $\rho$ to be over $S$ means that $\rho$ makes a commutative diagram with the projections from $A$ and $B$ to $S$, see Section \ref{sec:loci}. Some direct generalizations of Theorem \ref{paramrecti} will be stated as Theorem \ref{paramrectip} in a $p$-adic setting, and as Theorem \ref{paramrectimot} in a uniform $p$-adic setting.

\begin{proof}[Proof of Theorems \ref{p2} and \ref{p2'}]
We first prove the existence of $h_3$ as in (\ref{p3}), that is, we first prove the result for $\Iva(f,S)$.
Since the statement for $m=1$ can be applied successively, it is enough to prove the case $m=1$.  By Theorem \ref{paramrecti} and since Presburger functions are piecewise linear, there exists a finite partition of $S\times \ZZ$ and for each part $A$ a Presburger bijection $\rho:A\to B$ over $S$ such that either $B_x=\NN$ or $B_x$ is finite for each $x\in S$ and such that $f\circ \rho^{-1}$  is of the form
$$
(x,t)\mapsto \sum_{i=1}^r c_i(x) t^{a_i}q^{b_it}
$$
for some integers $a_i,b_i$ with $a_i\geq 0$, and some Presburger constructible functions $c_i$, and where the pairs $(a_i,b_i)$ are mutually different for different $i$. Denote the image of $A$ under the projection map $A\to S$ by $S_A$.
By Lemma \ref{lemqexp0} there exists a constant $M\geq 0$ such that, for each fixed value of $x$, either the $c_i(x)$ are all zero for $i=1,\ldots,r$, or, the function $t\mapsto \sum_{i=1}^r c_i(x) t^{a_i}q^{b_it}$ has at most $M$ zeros.
%By refining the partition of $S\times \ZZ$ by imposing conditions only on the variables running over $S$,
Write $S_{A,1}$ for the set of $x\in S_A$ such that $|B_x|\leq M$, and let $S_{A,2}$ be $S_{A}\setminus S_{A,1}$. Note that $S_{A,1}$ and $S_{A,2}$ are Presburger sets.
% $|B_x|>M$ for each $x$ in $A'$. %Let us write $p_A:A\to S$ for the restriction of the projection $S\times\ZZ\to S$ to $A$.
%In the case that $|B_x|\leq M$ for each $x$,
We take $M$ Presburger functions $H_1,\ldots, H_M$ on $S_{A,1}$ such that the union of the graphs of the $H_j$ equals
$B\cap (S_{A,1}\times \ZZ)$. We write
$$
Q_{A}:=\{x\in S_{A,1} \mid \bigwedge_{j\in \{1, \dots, M\}} (\sum_{i=1}^r c_i(x) H_j(x)^{a_i}q^{b_i H_j(x)} = 0) \},
$$
and
$$
R_{A}:=\{x\in S_{A,2} \mid \bigwedge_{i\in \{1, \dots, r\}} ( c_i(x) = 0 ) \}.
$$
By Lemma \ref{inters} each of the sets $Q_A$ and $R_{A}$ is the zero locus of a Presburger constructible function on $S$.
Now one has
$$
\Iva(f,S) = \bigcap\limits_{A} ( Q_A \cup R_A \cup (S\setminus S_A) ).
$$
By Lemma \ref{inters} and since the characteristic functions of the Presburger sets $S\setminus S_A$ are Presburger constructible, the existence of $h_3$ follows. This proves the existence of $h_3$ as in (\ref{p3}) for any given $f\in \cP_q(S\times \ZZ^m)$.

We use this result to prove simultaneously Theorem \ref{p2'} and the existence of $h_1$ and $h_2$.
The statements clearly allow us to partition $S \times \ZZ^m$ into finitely many pieces $A$
and to treat each one separately (for the existence of $h_1$ and $h_2$, this uses Lemma~\ref{inters}).
We choose a partition such that all Presburger functions involved in $f$ are
Presburger linear, we refine this partition using Theorem \ref{paramrecti}, and consider one resulting piece $A$. We can replace $A$ by $B$ and
$f$ by $f \circ \rho^{-1}$ with notation from Theorem \ref{paramrecti}, so that in the end, we get one Presburger set $B$ on which we have
\begin{equation}\label{frecti}
f (s,y) = \sum_{i=1}^r c_i(s) y^{a_i}q^{b_i\cdot y}
\end{equation}
where we use multi-index notation and where $a_i,b_i\in \ZZ^m$ with $a_{ij}\geq 0$, the $c_i$ are Presburger constructible functions in $s\in S$, the tuples $(a_i,b_i)$ are mutually different for different $i$, and where for each $s\in S$, one has $B_s=\Lambda_s\times \NN^\ell$ for a fixed $\ell\geq 0$
and
some finite set $\Lambda_s\subset \NN^{m-\ell}$ depending on $s$.
In fact, now we are already done by Lemma \ref{lemqcon0} and the existence of $h_3$ with $Z(h_3) = \Iva(g)$ for any given $g$. Indeed, let $I$ be $\{i\mid b_{ij}\geq 0\mbox{ for some } j = m-\ell +1,\ldots, m\}$. Consider the function on $B$
$$
h:(s,y)\mapsto \sum_{i\in I} c_i(s) y^{a_i}q^{b_i\cdot y}
$$
for $s\in S$ and $y\in B_s$. Let $\widetilde h$ be the extension by zero of $h$ to a function on $S\times \ZZ^{m}$.
By Lemma \ref{lemqcon0}, for $s\in S$, the family
$\{f  (s,y)\}_{y}$, where   $y \in B_s$, is summable if and only if $\widetilde h(s,y)=0$ for all $y\in\ZZ^{m}$. Since $\widetilde h$ is a Presburger constructible function on $S\times \ZZ^{m}$,
(\ref{p2eq}) follows using $h_1$ with $Z(h_1) = \Iva(\widetilde h)$. Taking $I'= \{i\mid ( b_{ij}> 0,\mbox{ or, } (b_{ij} =  0 \mbox{ and } a_{ij}>0))  \mbox{ for some } j = m-\ell +1,\ldots, m\}$ instead of $I$ in the above construction, one obtains the existence of $h_2$ for (\ref{p2b}) in a similar way.
Theorem \ref{p2'} also follows, since we can define $g$ piecewise for $(s,y)$ in $B$ by
$$
g (s,y) = \sum_{i\in \{1,\ldots,r\}\setminus I} c_i(s) y^{a_i}q^{b_i\cdot y}.
$$
\end{proof}

\begin{proof}[Proof of Theorem \ref{p1q-new}] By
the interpolation result Theorem \ref{p2'}, there exists $g_0$ in $\cP_q(S\times \ZZ^m)$ with $\Int(g_0,S)=S$ and such that
$f(s,y)=g_0(s,y)$ whenever $s$ lies in $\Int(f,S)$. Now, by Theorem-Definition 4.5.1 of \cite{CLoes}, the function $g$ that sends $s\in S$ to  $\sum_{y\in\ZZ^m} g_0(s,y)$ lies in $\cP_q(S)$. Clearly $g$ is as required.
\end{proof}

The above proof of Theorem \ref{p1q-new} uses Theorem-Definition 4.5.1 of \cite{CLoes}. For an alternative proof of Theorem \ref{p1q-new} which does not rely on \cite{CLoes}, %Alternatively to invoking Theorem-Definition 4.5.1 of \cite{CLoes} in the above proof of Theorem \ref{p1q-new},
one can proceed as follows. % in that proof to see that $g$ lies in $\cP_q(S)$.
Use Theorem \ref{paramrecti} to reduce to the case that $g$ is a sum as in the right hand side of  (\ref{frecti}), but with $m=1$.
If $y=y_1$ runs over $\NN$, one knows that the $b_i$ from (\ref{frecti}) are $<0$ by the proof of Theorem \ref{p2} and one uses explicit formulas for the summation of geometric power series and their derivatives. When $\Lambda_s\subset \ZZ$ is finite, with notation from (\ref{frecti}), one may
furthermore assume that $\Lambda_s$ is of the form $\{z\in\ZZ\mid 0\leq z\leq a(s)\}$, where $a$ is a positively valued Presburger function in $s$ and one uses geometric power series and their derivatives again to sum over $\Lambda_s$.

\subsection{Uniformity in the base $q$}\label{sec:qunif}

We show that the results of Section \ref{sec:qfix} hold uniformly in the base $q$. We will use this uniformity in the motivic setting.
Write $\RR_{>1}$ for $\{q\in\RR \mid q>1\}$.
\begin{defn}\label{defu}
As in \cite{CLoes}, define the subring $\AA\subset \QQ(\LL)$ as
 $$
\ZZ\left[\LL,\LL^{-1},\left(\frac{1}{1-\LL^{-i}}\right)_{i\in \NN,\ 0<i}\right],
 $$
where $\LL$ is a formal symbol. Each $a\in\AA$ is considered as a function
\begin{equation}\label{aq}
a: \RR_{>1}\to \RR:q\mapsto a(q)
\end{equation}
obtained by setting $\LL=q$.
For $S$ a Presburger set, let $\cP^u(S)$ be the $\AA$-algebra of $\RR$-valued functions on $S\times \RR_{>1}$ generated by
 \begin{enumerate}
\item the functions $\alpha:S \times \RR_{>1}\to \RR:(s,q)\mapsto \alpha(s)$ for all Presburger functions $\alpha:S\to\ZZ$,

\item the functions $q^\beta:S \times \RR_{>1}\to \RR:(s,q)\mapsto q^{\beta(s)}$ for all Presburger function $\beta:S\to\ZZ$.
\end{enumerate}
The functions in $\cP^u(S)$ are called Presburger constructible functions on $S$ with uniform base.
\end{defn}

The ring $\AA$ and a close variant of the rings $\cP^u(S)$ also appear in \cite{CLoes}.
The analogues of Theorems \ref{p2}, \ref{p2'}, and \ref{p1q-new} hold with almost the same proofs.

\begin{thm} [Correspondences of loci]\label{p2u} Let $S$ be a Presburger set and let $f$ be in $\cP^u(S\times \ZZ^m)$
for some $m\geq 0$.
Then there exist $h_1,h_2$ and $h_3$ in $\cP^u(S)$ such that
\begin{equation*}%\label{p2equ}
\Int(f,S\times \RR_{>1}) = Z(h_1), %\{(s,q)\in S\times \RR_{>1}\mid h_1(s,q)=0\},
\end{equation*}
\begin{equation*}%\label{p2bu}
\Bdd(f,S\times \RR_{>1}) = Z(h_2), %\{(s,q)\in S\times \RR_{>1}\mid h_2(s,q)=0\},
\end{equation*}
and
\begin{equation*}%\label{p3u}
\Iva(f,S\times \RR_{>1}) = Z(h_3), %\{(s,q)\in S\times \RR_{>1}\mid h_3(s,q)=0\}.
\end{equation*}
\end{thm}

\begin{thm}[Interpolation]\label{p2'u}
Let $f$ be in $\cP^u(S\times \ZZ^m)$ for some Presburger set $S$ and some $m\geq 0$.
Then there exists $g$ in $\cP^u(S\times \ZZ^m)$ such that $\Int(g,S\times \RR_{>1}) =  S\times \RR_{>1} $ and such that
$f(s,y,q)=g(s,y,q)$ whenever $(s,q)\in \Int(f,S\times \RR_{>1})$ and $y\in \ZZ^m$.
\end{thm}

\begin{thm}[Integration]\label{p1qu}
Let $f$ be in $\cP^u(S\times \ZZ^m)$ for some Presburger set $S$ and some $m\geq 0$. Then
there exists a function $g \in \cP^u(S)$ such that
$$
g(s,q) = \sum_{y\in\ZZ^m} f(s,y,q)
$$
whenever $(s,q)\in \Int(f,S\times \RR_{>1})$.
\end{thm}
\begin{proof}[Proof of Theorems \ref{p2u}, \ref{p2'u}, \ref{p1qu}]
Since the Lemmas \ref{lemqexp0} and \ref{lemqcon0} are completely uniform in $q$, the proofs of Section \ref{sec:qfix} go through almost literally the same way.
\end{proof}

\section{Integrability over a fixed $p$-adic field}\label{sec:qfixp}

In this section, we study the loci of Definition \ref{loci} and obtain similar results as in Section \ref{s:pres}, but now for functions on a finite degree field extension of $\QQ_p$.
We do this first in a setting without oscillation, and then in a setting where the functions may oscillate due to the presence of additive characters. The key technical result to control the difficulties related to oscillation is provided by Proposition \ref{intCLpexp}.

Let $K$ be a fixed finite field extension of $\QQ_p$ for a prime number $p$. Write $q_K$ for the number of elements in the residue field $k_K$ of $K$, and $\cO_K$ for the valuation ring of $K$ with maximal ideal $\cM_K$.
Fix $\cL_K$ to be either the semi-algebraic language on $K$ with coefficients from $K$, that is, Macintyre's language, or the subanalytic language on $K$ (as in e.g.~\cite{vdDHM} or \cite{Ccell}). Recall that Macintyre's language is the ring language $(\cdot,+,-,0,1)$ enriched with coefficients from $K$ and, for each integer $n>1$, a one variable predicate for the set of $n$-th powers in $K^\times$. The subanalytic language on $K$ is Macintyre's language enriched with the field inverse $^{-1}$ on $K^\times$ extended by $0^{-1}=0$, and for each convergent power series $f:\cO_K^n\to K$, a function symbol for the restricted analytic function
$$
x\in K^n\mapsto \begin{cases} f(x) & \mbox{ if }x\in\cO_K^n,\\
 0 & \mbox{ otherwise.}
    \end{cases}
$$
Note that one has quantifier elimination in $\cL_K$, by Macintyre's result \cite{Mac}, see also \cite{Denef2}, and by \cite{DvdD} for the subanalytic case.

Write $\varpi_K$ for a fixed uniformizer of $K$ and write $|\cdot|$ for the norm on $K$ with $|\varpi_K|=q_K^{-1}$. Put the normalized Haar measure on $K^n$, denoted by $|dx|$ whenever $x$ denotes a tuple of variables running over $K^n$, and where the normalization is such that $\cO_K^n$ has measure $1$.
For each integer $m>0$ consider the map $\ac_m:K\to \cO_K/ (\varpi_K^m)$ sending nonzero $x\in K$ to $\varpi_K^{-\ord x}\cdot x\bmod (\varpi_K^m)$ and sending $0$ to $0$. We also write $\ac$ for $\ac_1$.
To make the link with the motivic setting easier, we consider three sorted structures for our fixed $p$-adic field $K$. To this end, we enrich the language $\cL_K$ with the sorts $\ZZ$ for the value group, and $k_K$ for the residue field, together with the valuation map $\ord:K^\times\to \ZZ$ and the angular component map
$\ac$. Let us denote this three-sorted language by $\cL_K^3$.
Let us for each $m>1$ identify the map $\ac_m:K\to \cO_K/(\varpi_K^m)$ with a map $K\to k_K^m$, also denoted by $\ac_m$,
by using a bijection of finite sets $\cO_K/(\varpi_K^m)\to k_K^m$.

Endow $K^n\times k_K^m\times \ZZ^r$ with the product topology of the valuation topology on $K^n$ and the discrete topology on $k_K^m\times \ZZ^r$.
In this section, definable will mean $\cL_K^3$-definable.

\subsection{Constructible functions}

The ring of constructible functions $\cC(X)$ on a definable set $X$ is the $\AA_{q_K}$-algebra of real-valued functions on $X$ generated by functions of the form
\begin{enumerate}
\item $f:X\to\ZZ$ whenever $f$ is a definable function,

\item $q_K^g:X\to \AA_{q_K}:x\mapsto q_K^{g(x)}$ for definable functions $g:X\to \ZZ$.
\end{enumerate}
The functions in $\cC(X)$ are called constructible functions on $X$.

Now the analogues of Theorems \ref{p2}, \ref{p2'}, and \ref{p1q-new} hold in the $p$-adic setting.
In fact, the interest in the rings of constructible functions lies in their stability under integration, which we generalize to the following result.

\begin{thm}[Integration]
\label{p1p-new}
Let $f$ be in $\cC(X\times K^m)$ for some definable set $X$ and some $m\geq 0$. Then there exists $g\in \cC(X)$ such that
$$
g(x) = \int_{y\in K^m} f(x,y) |dy|
$$
whenever $x\in \Int(f,X)$.
\end{thm}

\begin{remark}\label{remCcell}
Under the extra condition that $\Int(f,X)=X$, Theorem \ref{p1p-new} was known: in the  subanalytic case this is Theorem 4.2 of \cite{Ccell} and the semi-algebraic case has the same proof as in \cite{Ccell}, using the semi-algebraic cell decomposition instead of the subanalytic cell decomposition. The first form of this kind of integration result (with some more conditions on $f$) goes back to the work by Denef in \cite{Denef1}, where the functions of $\cC(X)$ were introduced under a different name.
\end{remark}

\begin{thm}[Correspondences of loci]\label{p2p}
Let $f$ be in $\cC(X\times K^m)$ for some definable set $X$ and some $m\geq 0$.
Then there exist functions $h_1,h_2$ and $h_3$ in $\cC(X)$ such that the zero loci of $h_i$ equal respectively
\begin{equation*}%\label{p2eqp}
\Int(f,X), \qquad
\Bdd(f,X), \qquad\text{and}\qquad
\Iva(f,X),
\end{equation*}
for $i=1$, $2$, resp.~$3$, and with the normalized Haar measure on $K^m$.
\end{thm}

Theorem \ref{p2p} has the following corollary.
\begin{cor}\label{p2ploc}
Let $f$ be in $\cC(X\times K^m)$ for some definable set $X$ and some $m\geq 0$.
Then there exist functions $h_1$ and $h_2$ in $\cC(X)$ such that
\begin{equation}\label{p2eqploc}
\{x\in X\mid f(x,\cdot) \mbox{ is locally integrable on }K^m\} =  Z(h_1) % \{x\in X\mid h_1(x)=0\}
\end{equation}
and
\begin{equation}\label{p2bploc}
\{x\in X\mid f(x,\cdot) \mbox{ is locally bounded on }K^m\} = Z(h_2). %\{x\in X\mid h_2(x)=0\}.
\end{equation}
\end{cor}
\begin{proof}
Note that local integrability (and similarly for local boundedness) for a function $r$ on $K^n$ is equivalent to $\11_B\cdot r$ being integrable over $K^n$ (resp.~bounded on $K^n$) for each Cartesian product $B\subset K^n$ of balls in $K$, with characteristic function $\11_B$. Note that the family of all balls can be (possibly redundantly) realized as the members of a definable family (parameterized by, say, the radius and an element of the ball). Now the Corollary follows from the three statements of Theorem \ref{p2p}, where the existence of $h_3$ is used to eliminate the variables that were used to parameterize the balls.
\end{proof}

\begin{thm}[Interpolation]\label{p2'p}
Let $f$ be in $\cC(X\times K^m)$ for some definable set $X$ and some $m\geq 0$.
Then there exists $g$ in $\cC(X\times K^m)$ with $\Int(g,X)=X$ and such that
$f(x,y)=g(x,y)$ whenever $x$ lies in $\Int(f,X)$.
\end{thm}

The above results will be proved using the Cell Decomposition Theorem \ref{cd} below and the analogous results of Section \ref{sec:qfix}, but first we
state the main $p$-adic results in the exponential setting, which will have completely different and more difficult proofs.

\subsection{Constructible exponential functions}\label{osc}

Fix an additive character $\psi_K:K\to \CC^\times$ which is trivial on $\cM_K$ but nontrivial on $\cO_K$. (All characters are assumed to be unitary and continuous.)
The ring of constructible exponential functions $\cCexp(X)$ on a definable set $X$ is the $\AA_{q_K}$-algebra of complex-valued functions on $X$ generated by functions of the form
\begin{enumerate}%\item the constant functions with values in $\AA_{q_K}$,
\item $g$ with $g$ in $\cC(X)$;

\item functions $\psi_K(f):X\to \CC:x\mapsto \psi_K(f(x))$ for any definable function $f:X\to K$.
\end{enumerate}
The functions in $\cCexp(X)$ are called constructible exponential functions on $X$.

%The following analogue of Theorems \ref{p1q-new}, \ref{p1p-new} for constructible exponential functions is completely new. A much more restrictive %version can be found in \cite{CLexp} (see also \cite{HK}).

\begin{thm}[Integration]
\label{p1pexp}
Let $f$ be in $\cCexp(X\times K^m)$ for some definable set $X$ and some $m\geq 0$. Then there exists $g\in \cCexp(X)$ such that
$$
g(x)=\int_{y\in K^m} f(x,y) |dy|
$$
for all $x\in \Int(f,X)$.
\end{thm}
In \cite{CLexp}, Theorem \ref{p1pexp} is proved under an extra %condition
restriction: $f$ must be a finite sum of terms of the form $f_0\psi_K(f_1)$ with $f_1:X\times K^m\to K$ definable and $f_0\in \cC(X\times K^m)$ satisfying $\Int(f_0,X)=X$ (see also \cite{HK} for a similar result under a similar extra restriction).

We also find analogues of Theorems \ref{p2p} and \ref{p2'p} in the exponential setting.

\begin{thm}[Correspondences of loci]\label{p2pexp}
Let $f$ be in $\cCexp(X\times K^m)$ for some definable set $X$ and some $m\geq 0$.
Then there exist functions $h_1,h_2$ and $h_3$ in $\cCexp(X)$ such that
\begin{equation*}%\label{p2eqpexp}
\Int(f,X) = Z(h_1), %\{x\in X\mid h_1(x)=0\},
\end{equation*}
\begin{equation*}%\label{p2bpexp}
\Bdd(f,X) = Z(h_2), %\{x\in X\mid h_2(x)=0\},
\end{equation*}
and
\begin{equation*}%\label{p3pexp}
\Iva(f,X) = Z(h_3). %\{x\in X\mid h_3(x)=0\}.
\end{equation*}
\end{thm}

\begin{thm}[Interpolation]\label{p2'pexp}
Let $f$ be in $\cCexp(X\times K^m)$ for some definable set $X$.
Then  there exists $g$ in $\cCexp(X\times K^m)$ with $\Int(g,X)=X$ and such that
$f(x,y)=g(x,y)$ whenever $x$ lies in $\Int(f,X)$. Moreover, one can write any such $g$ as a finite sum of terms of the form
$$
f_0\psi_K(f_1)
$$
with $f_1:X\times K^m\to K$ definable and $f_0\in \cC(X\times K^m)$ satisfying $\Int(f_0,X)=X$.
\end{thm}

Theorem \ref{p2pexp} implies the following corollary by the same reasoning as for Corollary \ref{p2ploc}.
\begin{cor}\label{p2plocexp}
Let $f$ be in $\cCexp(X\times K^m)$ for some definable set $X$ and some $m\geq 0$.
Then there exist functions $h_1$ and $h_2$ in $\cCexp(X)$ such that
\begin{equation}\label{p2eqplocexp}
\{x\in X\mid f(x,\cdot) \mbox{ is locally integrable on }K^m\} = Z(h_1) %\{x\in X\mid h_1(x)=0\}
\end{equation}
and
\begin{equation}\label{p2bplocexp}
\{x\in X\mid f(x,\cdot) \mbox{ is locally bounded on }K^m\} = Z(h_2). %\{x\in X\mid h_2(x)=0\}.
\end{equation}
\end{cor}

The following key technical Proposition excludes strange oscillatory behavior of exponential constructible functions.
This will allow us to reduce to the techniques and results from the previous sections.

\begin{prop}
\label{intCLpexp}%[IML]
Let $m\geq 0$ and $s\geq 0$ be integers,
let $X$ and $U \subset X \times K^{m}$ be definable,
and let $f_1, \dots, f_s$
be in $\cCexp(U)$. Write $x$ for variables running over $X$ and $y$ for variables running over $K^m$.
Then there exists an integer $d\geq 0$, a definable surjection $\varphi:U\to V\subset X\times \ZZ^{t}$ over $X$ for some $t\geq 0$, definable functions $h_{\ell i}:U\to K$, and functions $G_{\ell i}$ in $\cCexp(V)$ for $\ell=1, \dots, s$,
such that the following conditions hold.
 \begin{itemize}
 \item[1)] For each $\ell$ with $1\le\ell\le s$,  one has $$f_\ell(x,y) = \sum_{i=1}^{N_\ell}G_{\ell i}(\varphi(x,y))\psi_K(h_{\ell i}(x,y)),$$
 for some positive integer $N_\ell$;
 \item[2)] if one sets, for $(x,r)\in V$,
$$
U_{x,r} := \{y\in U_x\mid \varphi(x,y)=(x,r)\}
$$
 and
 $$
 W_{x,r}:=\{y \in U_{x,r}\mid
 \sup_{\ell,i} | G_{\ell i}(x,r)|_{\CC} \leq \sup_{\ell} |f_\ell(x,y)|_{\CC} \},
 $$
where $|\cdot|_\CC$ is the complex modulus, then
\[
\Vol(U_{x,r}) \leq q_K^d \cdot \Vol(W_{x,r}) < +\infty,
\]
where the volume $\Vol$ is taken with respect to the Haar measure on $K^m$.
\end{itemize}
\end{prop}

Roughly, the proposition for $s=1$ says that, if $|f_1|_\CC$ is small, then $f_1$ is the sum of small terms of a very specific form.
Indeed, for the functions
$|G_{1 i}(\varphi(x,\cdot))\psi_K(h_{1 i}(x,\cdot))|_{\CC}=|G_{1 i}(\varphi(x,\cdot))|_{\CC}$ to be small, it suffices to know that they are small on the sets $W_{x,r}$, since they are constant on each superset $U_{x,r}$.
More precisely, if $f_1$ can not be written as a sum of small terms
as in 1), then $|f_1|_{\CC}$ has to be large on a relatively large set, namely on the set $W_{x,r}$.

For the proposition to make sense, the sets $U_{x,r}$ and $W_{x,r}$ have to be
measurable, but this follows from the facts that each definable set is measurable, functions in $\cCexp(Z)$ are measurable for any definable $Z$, and, that the space of measurable functions is closed under
taking the complex modulus and the supremum.
We prove this proposition in the next section. First, we need to set up some preliminaries.

\subsection{Preliminaries for the $p$-adic proofs}\label{prelimp}
We give a notion of $p$-adic cells which is adapted to the three sorts in $\cL_K^3$ and which fits better
with the motivic approach below than some previously used notions of $p$-adic cells, see especially the usage of $\xi$ in the next definition.

 \begin{defn}[$p$-adic cells]\label{def::cell}
Let $Y$ be a definable set.
A $1$-cell $A\subset Y\times K$ over $Y$ is a (nonempty) set of
the form
 \begin{eqnarray}\label{cel}
A & = & \{(y,t)\in Y'\times K\mid \alpha(y) \sq_1 \ord(t-c(y)) \sq_2 \beta(y), \\
& &  \ord (t-c(y))\in a+n\ZZ,\
  \ac_m (t-c(y)) = \xi(y) \}, \nonumber
\end{eqnarray}
with $Y'$ a definable subset of $Y$,
integers $a\geq 0$, $n>0$, $m>0$, $\alpha,\beta\colon Y'\to \ZZ$ and $\xi:  Y'\to  (\cO_K/ (\varpi_K^m))^\times$ definable, $c\colon Y'\to K$ definable, and $\sq_i$ either $<$ or no
condition, and such that $A$ projects surjectively onto $Y'$.
We call $c$ the center, $\xi$ the angular component,
$a+n\ZZ$ the coset, $\alpha$ and $\beta$ the boundaries,
and $Y'$ the
base of $A$. A $0$-cell $A\subset Y\times K$ over $Y$ is a (nonempty) set of
the form
\begin{eqnarray}\label{0cel}
A & = & \{(y,t)\in Y'\times K\mid t=c(y)\},
\end{eqnarray}
with $Y'$ a definable subset of $Y$,
and $c\colon Y'\to K$ definable.
In both cases we call $A$ a cell over $Y$ with center $c$.
\end{defn}

Also our formulation of the cell decomposition result is somehow more relaxed than usual, due to having three sorts.
For a slightly stronger cell decomposition result than Theorem \ref{cd} and references, see e.g.~\cite{CCLLip}, Theorem 3.3.
\begin{thm}[$p$-adic Cell Decomposition \cite{Denef2}, \cite{Ccell}]\label{cd}
Let $X\subset Y\times K$ and $f_j:X\to \ZZ$ be definable for some definable set $Y$ and $j=1,\ldots,r$. Then there exists a finite partition of $X$ into
cells $A_i$ (over $Y$) with center $c_i$ such that for each occurring $1$-cell $A_i$ with coset
$a_i+n_i\ZZ$ and base $Y'_i$ one has
 \begin{equation*}
 f_j(y,t)=
 h_{ij}(y) + a_{ij} \frac{ \ord (t-c_i(y))- a_i}{n_i} ,\quad
 \mbox{ for each }(y,t)\in A_i,
 \end{equation*}
with integers $a_{ij}$ and
$h_{ij}:Y'_i\to \ZZ$ definable functions for $j=1,\ldots,r$. Moreover, if also $g_j\in\cC(X)$ are given for $j=1,\ldots,r'$, then the cells $A_i$ can be taken such that, for each $y\in y$, the function $g_i(y,\cdot)$ is constant on each ball contained in $A_{i y}$.
 \end{thm}

For $X\subset K$ open, a function $f\colon X\to K$ is called $C^1$ if $f$ is differentiable at each point of $X$ and the derivative $f'\colon X\to K$ of $f$ is continuous. (This notion of $C^1$, although more naive than the ones in e.g.~\cite{Glock}, suffices for our purposes.)
 A ball in $K$ is by definition a set of the form
$
\{t \in K\mid \ord ( t - a ) \geq z\}
$
for some $a\in K $ and some $z \in \ZZ$. For $X$ a subset of $K$, by a maximal ball contained in $X$ we mean a ball $B\subset X$ which is maximal with respect to inclusion among all balls contained in $X$.

\begin{defn}[Jacobian property]\label{defjacprop}
Let $f\colon B_1\to B_2$ be a function with
$B_1,B_2\subset K$. Say that $f$
\textit{has the Jacobian property} if the following
conditions a) up to d) hold:
\begin{itemize}

\item[a)] $f$ is a bijection from $B_1$ onto $B_2$ and $B_1$ and $B_2$ are balls in $K$;

\item[b)] $f$ is $C^1$ on $B_1$ with nonvanishing derivative $f'$;

\item[c)] $|f'|$ is constant on $B_1$;

\item[d)] for all $x,y\in B_1$ one has
$$
|(x-y)\cdot f'|=|f(x)-f(y)|.
$$
\end{itemize}
\end{defn}

\begin{defn}[$1$-Jacobian property]\label{def1jacprop}
Let $f\colon B_1\to B_2$ be a function with
$B_1,B_2\subset K$. Say that $f$
\textit{has the $1$-Jacobian property} if $f$ has the Jacobian property
and moreover e) and f) hold:
\begin{itemize}
\item[e)] $\ac(f')$ is constant on $B_1$;

\item[f)] for all $x,y\in B_1$ one has
$$
\ac (f')\cdot \ac(x-y)=\ac(f(x)-f(y)).
$$
\end{itemize}
\end{defn}

The following two results will be important for the proofs in the exponential setting.

\begin{prop}[\cite{CLip}, Section 6]\label{jacprop}
Let $Y$ and $X\subset Y\times K$ be definable sets and let
$F:X\to K$ be a definable function.
Then there exists a finite partition of $X$ into cells $A_i$ over $Y$ such that for each $i$
and each $y\in Y$, the restriction of
$F(y,\cdot)$
to $A_{iy}:=\{t\in K\mid (y,t)\in A_i\}$
is either constant or injective, and such that in the latter case,
for each
ball $B\subset K$ such that $\{y\}\times B$ is contained in $A_i$, there is a ball
$B^* \subset K$
such that $F(y, B) = B^*$ and such that the map
$$
F_B:B\to B^*:t\mapsto F(y,t)
$$
has the $1$-Jacobian property.
 \end{prop}

\begin{lem}\label{avoidSkol}
Let $A\subset Y\times K$ and $h:A\to K$ be definable for some definable set $Y$. Suppose that for each $y\in Y$, and for each maximal ball $B$ contained in $A_y$, $h(y,\cdot)$ is constant modulo $(\varpi_K)$ on $B$. Then there exists a finite partition of $A$ into definable sets $A_j$ and definable functions $h_{j}:Y\to K$ such that
$$
|h(y,t)-h_{j}(y)| \leq 1
$$
holds for each $(y,t)\in A_j$ and for each $j$.
\end{lem}
\begin{remark}
We state the lemma in its present form and prove it in such a way that makes it easy to adapt to the motivic case below.
\end{remark}
\begin{proof}[Proof of Lemma~\ref{avoidSkol}]
The proof is similar to the proof of Theorem 2.2.2 of \cite{CLexp} and goes as follows. Up to a finite partition of $A$, we may suppose that for each $y\in Y$, the function $h(y,\cdot)$ is injective (see e.g.~Corollary 3.7 of \cite{CCLLip}). Similarly we may suppose that $h$ is as $F$ in the conclusion of Proposition \ref{jacprop} already on the whole of $A$. We may moreover for each two balls $\{y\}\times B_1$ and $\{y\}\times B_2$ contained in $A$ assume that the images $h(\{y\}\times B_1)$ and $h(\{y\}\times B_2)$ are balls with different radii, for example by invoking Theorem \ref{cd}.
Now consider the graph of $h$ in $A\times K$, and its image $W\subset Y\times K$ under the coordinate projection sending $(y,t,h(y,t))$ to $(y,h(y,t))$.
Take a cell decomposition as in Theorem \ref{cd} of $W$ into cells $C_j$.
Write $c_j$ for the center of $C_j$. Let us fix $m_j>0$ such that the angular component of $C_j$ (i.e., the map denoted by $\xi$ in Definition~\ref{def::cell}) takes values in $\cO_K/ (\varpi_K^{m_j})$ in the case that $C_j$ is a $1$-cell, and such that $m_j=1$ if $C_j$ is a $0$-cell. Now it follows for all $(y,t)\in A$ with $h(y,t)\in C_j$ that
$$
|h(y,t)-c_j(y)| \leq q_K^{m_j-1}.
$$
Let $A_j$ be the set consisting of $(y,t)\in A$ with  $h(y,t)\in C_j$. If $m_j=1$ we are done
by taking $h_j=c_j$ on $A_j$. If $m_j>1$ we can finish by further partitioning using the finiteness of the residue field.
\end{proof}

\subsection{The $p$-adic proofs for constructible functions}\label{proofconstrp}
For $p$-adic constructible functions (thus without additive character), the proofs reduce to the Presburger cases of Section \ref{sec:qfix} with $q=q_K$, via $p$-adic cell decomposition and the following definitions and results.

\begin{defn}\label{full}
If $f_j\colon X \subset K^{m+1} \to \ZZ$ and the $A_i$ are as in Theorem \ref{cd}, then
call $f_j$ \emph{prepared} on the cells $A_i$.
We call $A_i$ a \emph{full cell} and we call $f_j$ \emph{fully prepared} on the $A_i$
if the base of $A_i$ is itself a cell on which the $h_{ij}(x)$---with notation from Theorem \ref{cd} for $1$-cells and with $h_{ij}=f_j$ in the case of $0$-cells---and the boundaries of $A_i$ are prepared, and so on $m$ times.
It is also clear what we mean by a full cell $A\subset Y \times K^{m+1}$ \emph{over} some definable set $Y$, in analogy to the notion of cells over $Y$ of Definition \ref{def::cell}.
By the centers of a full cell, we mean a tuple of centers, consisting of the center of the cell $A$ over $Y$, the center of the base $A'$ of $A$, the center of the base of $A'$ and so on.
\end{defn}

\begin{defn}\label{skeleton}
Let $A\subset K^{m}$ be a full cell with center $c_{m}$ in the last coordinate, up to center $c_1\in K$ for the first coordinate. The skeleton of $A$ is then the subset $S(A)$ of $(\ZZ\cup\{+\infty\})^{m}$ which is the image of $A$ under the map
$$
x\in A\mapsto (\ord (x_1-c_1 )  ,\ord (x_2-c_2(x_1) )  , \ldots, \ord (x_m-c_m(x_1,\dots,x_{m-1}) ) ,
$$
where we have extended $\ord$ to a map $\ord: K\to \ZZ\cup\{+\infty\}$.
Write $s_A$ for the natural map $A\to S(A)$ which we call the skeleton map.
Likewise, if $A\subset Y\times K^\ell$ is a full cell over $Y$ for some definable set $Y$, it is clear what we mean by the skeleton $S_{/Y}(A)$ over $Y$ and the skeleton map $s_{A/Y}$ of $A$ over $Y$. \end{defn}

\begin{remark}\label{skel}
For $A\subset K^{m}$ a full cell, any fiber of the skeleton map $s_A$ is a Cartesian product of singletons and balls, and hence, the volume of such a fiber $s_A^{-1}(r)$ equals $0$ for each $r$, or, equals $q_K^{\alpha(r)}$ for a definable function $\alpha:A\to \ZZ$. Clearly the set of $r$ such that
$s_A^{-1}(r)$ has volume zero is a definable set.
\end{remark}

Call a definable function $f:X\subset Y\times \ZZ^m\to Y\times \ZZ^\ell$ linear over $Y$ if there is a definable function $a:Y\to \ZZ^\ell$ and an affine map $g:\QQ^m\to\QQ^\ell$ such that $f(y,z) = (y,g(z) + a(y))$ for all $(y,z)\in X$.
\begin{prop}[Parametric Rectilinearization for $\cL_K^3$]\label{paramrectip}
Let $Y$ and $X\subset Y\times \ZZ^m$ be definable sets. Then there exists a finite partition of $X$ into
definable sets such that for each part $A$, there is a set $B\subset Y\times \ZZ^{m}$
and a definable bijection $\rho:A\to B$ which is linear over $Y$
such that, for each $y\in Y$, the set
$B_y$
is
a set of the form
$\Lambda_y\times \NN^\ell$ for a finite subset $\Lambda_y\subset \NN^{m-\ell}$ depending on $y$ and for an integer $\ell\geq 0$ only depending on
$A$.
\end{prop}
\begin{proof}
This follows from Theorem \ref{paramrecti} as follows. If $Y\subset \ZZ^r$ then this is Theorem \ref{paramrecti}. If  $Y\subset \ZZ^r \times k_K^n$, then the result follows from Theorem \ref{paramrecti} and orthogonality between the $k_K$-sort and the $\ZZ$-sort. (This orthogonality is the property that any definable subset of $\ZZ^r \times k_K^n$ equals a finite union of Cartesian products of definable subsets of $\ZZ^r$ and of $k_K^n$; this property follows from quantifier elimination of $K$-variables.)
Now suppose that $Y\subset \ZZ^r\times K^\ell\times k_K^n$.  By quantifier elimination of $K$-variables, there exist $s$ and $t$, a definable function
$$
\nu:Y\to k_K^t \times \ZZ^s,
$$
and a definable set
$$X'\subset  k_K^t\times \ZZ^{s+m},
$$
such that $X_y$ equals $X'_{\nu(y)}$ for every $y\in Y$. Applying the previous case to the set $X'$ yields a partition of $X'$ into parts $A'$ with the desired properties. Let $\kappa$ be the function $$
\kappa:X\to X': (y,w)\mapsto (\nu(y),w).
$$
Now the sets $\kappa^{-1}(A')$ form a partition of $X$ with the desired properties. (As an alternative proof one may adapt the proof of Theorem 3 of \cite{CPres}.)
\end{proof}

We now have the following variants of the Presburger Theorems \ref{p2} and \ref{p2'}.
\begin{cor}\label{p2pZ}
Let $f$ be in $\cC(X\times \ZZ^m)$ for some definable set $X$ and some $m\geq 0$.
Then there exist $h_1,h_2$ and $h_3$ in $\cC(X)$ such that
\begin{equation}\label{p2eqpZ}
\Int(f,X) = Z(h_1), %\{x\in X\mid h_1(x)=0\},
\end{equation}
\begin{equation}\label{p2bpZ}
\Bdd(f,X) = Z(h_2), % \{x\in X\mid h_2(x)=0\},
\end{equation}
and
\begin{equation}\label{p3pZ}
\Iva(f,X) = Z(h_3), %\{x\in X\mid h_3(x)=0\},
\end{equation}
where integrability in (\ref{p2eqpZ}) is with respect to the counting measure on $\ZZ^m$.
\end{cor}
\begin{cor}[Interpolation]\label{p2'pZ}
Let $f$ be in $\cC(X\times \ZZ^m)$ for some definable set $X$ and some $m\geq 0$.
Then there exists $g$ in $\cC(X\times \ZZ^m)$ with $\Int(g,X)=X$ and such that
$f(x,y)=g(x,y)$ whenever $x$ lies in $\Int(f,X)$.
\end{cor}
\begin{proof}[Proofs of Corollaries \ref{p2pZ} and \ref{p2'pZ}]
The proofs of Theorems \ref{p2} and \ref{p2'} apply also in this setting, where one uses Theorem \ref{paramrectip} instead of Theorem \ref{paramrecti}.
\end{proof}

\begin{proof}[Proof of Theorems \ref{p2p} and \ref{p2'p}]
By an inductive application of the Cell Decomposition Theorem \ref{cd}, partition $X\times K^m$ into finitely many full cells $A_i$ over $X$ such that for each $i$, the restriction $f_{|A_i}$ factors through the skeleton map $s_{A_i/X}$ of $A_i$ over $X$. Let us identify each skeleton with a definable set, for example by replacing $\{+\infty\}$ by a disjoint copy of the singleton $\{0\}$.
Let us write $f_i$ for the map from the skeleton $S_{/X}(A_i)$ of $A_i$ over $X$ to $\AA_{q_K}$ induced by $f_{|A_i}$. Then $f_i$ lies in $\cP_{q_K}(S_{/X}(A_i))$ for each $i$.
The function $M$ sending $z\in S_{/X}(A_i)$ to the volume of the fiber $(s_{A_i/X})^{-1}(z)$, taken inside $K^m$, either is identically zero on $S_{/X}(A_i)$ or is of the form $q_K^{\alpha_i(z)}$ on $S_{/X}(A_i)$ for some definable function $\alpha_i$,
since the fibers of $s_{A_i/X}$ are of a very simple form, see Remark \ref{skel}. Hence, $\tilde f_i$ given by $z\mapsto f_i(z)\cdot M(z)$ lies in $\cC(S_{/X}(A_i))$.
For each $i$, let $X_i$ be the image of $A_i$ under the projection to $X$. Note that $X_i$ also equals the image of $S_{/X}(A_i)$ under the projection to $X$.
By Corollary \ref{p2pZ}, we now obtain for each $i$ that there are $h_{i1}$, $h_{i2}$ and $h_{i3}$ in
$\cC(X_i)$ such that
$$
\Int(\tilde f_i,X_i) = Z(h_{i1}), %\{x\in X_i\mid h_{i1}(x)=0\},
$$
$$
\Bdd(f_i,X_i)= Z(h_{i2}), %\{x\in X_i\mid h_{i2}(x)=0\},
$$
and
$$
\Iva(f_i,X_i) = Z(h_{i3}). %\{x\in X_i\mid h_{i3}(x)=0\}.
$$
Extend each of the $h_{ij}$ by zero to a function ${\tilde h}_{ij}$ in $\cC(X)$, thus defined on the whole of $X$. Now for $j=1,2,$ or $3$, let $h_j$ be the function
$$
\sum_{i}{\tilde h}_{ij}^2.
$$
Then the $h_j$ for $j=1,2,3$ are as required.

For the construction of $g$ as desired for Theorem \ref{p2'p} one proceeds as follows.
By Corollary \ref{p2'pZ} one finds for each $i$ a function $g_i$
in $\cC(S_{/X}(A_i))$ with $\Int(g_i,X_i)=X_i$ and such that
$\tilde f_i(x,s)=g_i(x,s)$ whenever $x$ lies in $\Int(\tilde f_i,X_i)$.
Now we define the function $g$ on $X\times K^m$ as follows.
For $(x,y)\in A_i$, let $z=s_{A_i/X}(x,y)$, and let $g(x,y):=f(x,y)$ if $M(z)=0$, and the quotient $g(x,y):=g_i(z)/M(z)$ when $M(z)\not=0$. Note that constructible functions stay constructible when divided by $M$ whenever $M$ is nonzero, by its simple form $q_K^{\alpha_i}$ as described above. The function $g$ is as required.
\end{proof}

\begin{proof}[Proof of Theorem \ref{p1p-new}]
The theorem follows from Theorem \ref{p2'p} and \cite[Theorem 4.2]{Ccell}, in the same way as Theorem \ref{p1q-new} follows from Theorem \ref{p2'} and \cite[Theorem-Definition 4.5.1 ]{CLoes}. Indeed, by
the interpolation result Theorem \ref{p2'p}, there exists $g_0$ in $\cC(X\times K^m)$ with $\Int(g_0,X)=X$ and such that
$f(x,y)=g_0(x,y)$ whenever $x$ lies in $\Int(f,X)$. Now, by \cite[Theorem 4.2]{Ccell} and by Remark \ref{remCcell} for the semi-algebraic case, the function $g$ which sends $x\in X$ to  $\int_{K^m}g_0(x,y)|dy|$ lies in $\cC(X)$. Clearly $g$ is as required.
\end{proof}

In fact, the proof of Theorems \ref{p2p} and \ref{p2'p} yields the following slightly more general variant.
\begin{cor}\label{corHX}
Let $f$ be in $\cH(X)\otimes_{\cC(X)} \cC(X\times K^m)$ for some definable set $X$, some $m\geq 0$, and some inclusion $\cC(X)\subset \cH(X)$ of $\AA_q$-algebras of complex valued functions on $X$.
Then there exist functions $h_1,h_2$, $h_3$ in $\cH(X)$ and $g$ in $\cH(X)\otimes_{\cC(X)} \cC(X\times K^m)$  such that the zero loci of the $h_i$ equal respectively
\begin{equation*}%\label{p2eqp}
\Int(f,X),\qquad
\Bdd(f,X),\qquad\text{and}\qquad
\Iva(f,X),
\end{equation*}
and such that $\Int(g,X)=X$ and
$f(x,y)=g(x,y)$ whenever $x$ lies in $\Int(f,X)$. Moreover, any such $g$ can be written as a finite sum of terms of the form
$$
h_i\cdot f_i
$$
with $h_i\in \cH(X)$ and $f_i\in \cC(X\times K^m)$ satisfying $\Int(f_i,X)=X$.
\end{cor}

\subsection{The $p$-adic proofs for constructible exponential functions}\label{proofexpp}

Consider a finite field $\FF_q$ with a nontrivial additive character $\psi$.
The following lemma and its corollary are classical exercises.
\begin{lem}\label{four}
For any function $f:\FF_q\to\CC$ one has
$$
\frac{1}{q} \|\hat f\|_{\sup} \leq \| f\|_{\sup} \leq \|\hat  f\|_{\sup}
$$
where $\|\cdot\|_{\sup} $ is the supremum norm and $\hat f$ the Fourier transform of $f$,
$$
\hat f ( y ) = \sum_{x\in\FF_q} f(x)\psi(-xy).  %  \exp(\frac{-2\pi i}{p} {\rm Trace}_{\FF_q/\FF_p}(xy)).
$$
\end{lem}

\begin{cor}\label{linear}
Consider a function
$$f:\FF_q\to\CC:y\mapsto \sum_{j=1}^s  c_j \psi(b_jy)$$
for some complex numbers $c_j$ and some distinct $b_j\in\FF_q$. Then there exists $y_0\in \FF_q$ with
$$
 \sup_{j=1}^s |c_j|_{\CC} \leq |f(y_0)|_{\CC}.
$$
\end{cor}

\begin{lem}\label{conju}
Let $X$ be a definable set and
let $f$ be in $\cCexp(X)$. Then there exists a function $g$ in $\cCexp(X)$ such that for each $x\in X$,
$f (x)$ and $g (x)$ are conjugate complex numbers. In particular, $f (x)g (x)$ equals the square of the complex norm of $f (x)$ for each $x\in X $, and thus in particular $Z(f) = Z(fg)$.
\end{lem}
\begin{proof}
The function $g$ can be obtained from $f$ by putting a minus sign in each of the arguments of the additive characters which occur in $f$. More precisely, write $f$ as a finite sum of terms of the form $g_i\psi_K(h_i)$ for $g_i$ in $\cC(X)$ and definable functions $h_i:X\to K$, and define $g$ as the corresponding sum with terms $g_i\psi_K(-h_i)$. Since the $g_i$ take real values, $g$ is as desired.
\end{proof}
This Lemma yields that any finite intersection or finite union of zero loci of functions in $\cCexp(X)$ is again a zero locus of a function in $\cCexp(X)$, analogous to Lemma \ref{inters}.
\begin{cor}\label{interspexp}
Let $X$ be a definable set and let $h_i$ be in $\cCexp(X)$ for $i=1,\ldots,N$.
Then there exist $f$ and $g$ in $\cCexp(X)$ such that
$$
Z(f) = \bigcap_{i=1}^{N} Z(h_i)\ \mbox{ and }\ Z(g) = \bigcup_{i=1}^{N} Z(h_i).
$$
The corresponding statement for $h_i$ in $\cC(X)$ also holds, yielding $f,g\in\cC(X)$.
\end{cor}
\begin{proof}
For $f$ one can take $\sum_{i=1}^N h_i \overline{h_i}$, where $\overline{h_i}$ is the complex conjugate of $h_i$ as given by Lemma \ref{conju}. For $g$ one simply takes the product of the $h_i$. For $h_i$ in $\cC(X)$ one takes the sum of the squares of the $h_i$ for $f$ and the product for $g$.
\end{proof}
As for Presburger constructible functions, the complement of $Z(h)$
in $X$ for $h\in\cC(X)$ is not always equal to the zero locus of some function in $\cC(X)$, and similarly for $h$ in $\cCexp(X)$.

\begin{proof}[Proof of Proposition~\ref{intCLpexp} for $m=1$]%[IML]
The statement that we have to prove clearly allows us to work
piecewise; if we have a finite partition of $U$ into definable parts $A$, then it suffices to prove the proposition for $f_\ell$ restricted to each part $A$. We actually prove something slightly stronger than
Proposition~\ref{intCLpexp} for the case $m=1$. That is, for a given definable function $\varphi_0:U\to X\times \ZZ^{t_0}$ over $X$, we prove that in addition to the conclusions 1) and 2) of the
proposition, we can require that also the following conditions 3) and 4) hold for each $x\in X$.
\begin{itemize}
\item[3)]
Each of the sets $U_{x,r}$ is either a singleton, or, equal to a maximal ball contained in $U_x$.
\item[4)]\label{cond4} The function $\varphi_0$ factors through $\varphi$, that is,  $\varphi_0= \theta\circ \varphi$ for some definable function $\theta$. %:V\to  X\times \ZZ^{t_0}$
\end{itemize}

So, let a definable function $\varphi_0:U\to X\times \ZZ^{t_0}$ over $X$ be given. By definition of $\cCexp$, the $f_\ell$ are finite sums of terms of the form $g\psi_K(h)$ for some $g\in\cC(U)$ and some definable $h:U\to K$.
Apply Theorem \ref{cd} to these functions $g\in\cC(U)$ and to the last $t_0$ component functions of $\varphi_0$. By working piecewise, we may suppose that $U$ is one of the so-obtained cells. If $U$ is a $0$-cell over $X$, there is nothing to prove. If $U$ is a $1$-cell over $X$, say, with center $c$, then let $\varphi_1:U\to V_1\subset X\times \ZZ$ be the definable surjective function sending $(x,y)$ in $U$ to $(x,\ord(y-c(x))$. It follows that there are definable functions $h_{\ell i}:U\to K$ and functions $G_{\ell i}$ in $\cCexp(V_1)$ such that for each $\ell$ one has
\begin{equation}\label{f-as-sum}
f_\ell(x,y) = \sum_{i=1}^{N_\ell}G_{\ell i}(\varphi_1(x,y))\psi_K(h_{\ell i}(x,y)),
\end{equation}
and that, for each $x$, the collection of the sets $U^{1}_{x,r} := \{y\in U_x \mid \varphi_1(x,y)=(x,r)\}$
equals the collection of maximal balls contained in $U_x$. Thus, the conditions 1), 3) and 4) already hold for $\varphi_1$.
We now construct $\varphi$ (and modify $G_{\ell i}$ and $h_{\ell i}$ accordingly) such that moreover 2) holds.  %, and such that $\varphi$ factors through $\varphi_1$.

We will proceed by induction on $N := \sum_{\ell=1}^s (N_\ell - 1)$.
Namely, assume that for any finite family of functions $\{f_\ell\}$ on a definable set $U$ (not necessarily the same family and the same set as the given one), such that
the functions $f_\ell$ have a presentation of the form (\ref{f-as-sum}) and satisfying the properties 1), 3), and 4), and with $\sum (N_{\ell}-1)<N$, there exists a function $\varphi$ such that the property 2) holds as well. Then we want to prove the same for any such family and presentation with $\sum(N_\ell-1)=N$. The idea of the proof of the induction step is to increase the number of functions in the family without increasing the total number of terms in their presentations (\ref{f-as-sum}), and thus decrease
$\sum(N_\ell -1)$.  To achieve this, one is still allowed to work piecewise, i.e., to replace $U$ with a subset that comes from using cell decomposition as above. Note that the constant $d$ appearing in 2) will increase by
at most $1$ in each induction step, so that we actually obtain $d \le N$.

If $N = 0$, then all $N_\ell = 1$, and one is done, taking $\varphi=\varphi_1$ and $d = 0$. Indeed, if $N_\ell = 1$, then $|G_{\ell 1}(x,r)|_{\CC}$ equals $|f_\ell(x,y)|_{\CC}$, and thus, if $N=0$, then  $U_{x,r}=W_{x,r}$.

For general $N>0$ we start by pulling out the factor $\psi_K(h_{\ell 1}(x,y))$ out of (\ref{f-as-sum}), i.e., we may assume that $h_{\ell 1}(x,y)=0$ for all $\ell$ and all $(x,y) \in U$.
By Proposition \ref{jacprop} we may moreover suppose that
for each $(x,r)\in V_1$, each $\ell$,  and each $i$ either $h_{\ell i}(x,\cdot)$ is constant on $U^1_{x,r}$, or $h_{\ell i}(x,\cdot)$ restricted to  $U^1_{x,r}$ has the $1$-Jacobian property.
Hence, for each $(x,r)\in V_1$ there exist constants $b_{x,r,\ell,i}\in  K$ such that, for all $y_1,y_2\in U^1_{x,r}$ and all $\ell,i$,
\begin{eqnarray}
  \ord(h_{\ell i}(x,y_1) - h_{\ell i}(x,y_2))
  & = &
  \ord(b_{x,r,\ell,i}\cdot (y_1-y_2)), \label{jac1} \\
  \ac (h_{\ell i}(x,y_1) - h_{\ell i}(x,y_2))
  &=&
  \ac(b_{x,r,\ell,i}\cdot (y_1-y_2)) \label{jac2},
\end{eqnarray}
where $b_{x,r,\ell,1} = 0$ by a previous assumption, and where we write $\ord: K\to \ZZ\cup\{+\infty\}$.
 If for all $\ell, i, x,r$, the function $h_{\ell i}(x,\cdot)$ is constant modulo $(\varpi_K)$ on $U^1_{x,r}$, then, up to a further finite partition of $U$, Lemma \ref{avoidSkol} applied to each of the $h_{\ell i}$ brings us back to the case  $N=0$.
We may thus in particular assume that for each $(x,r)$ in $V_1$,  there exist $\ell, i$ with $b_{x,r,\ell,i} \ne 0$. Choose $\gamma_{x,r} \in K$ with
\[
|\gamma_{x,r}|\cdot \max_{\ell,i} |b_{x,r,\ell,i}| = 1.
\]
For each $x$, $r$ and $\ell$, partition $\{1, \dots, N_\ell\}$ into non-empty subsets $S_{\ell j}(x,r)$, $j \ge 1$, with the property that $i_1,i_2$ lie in the same part $S_{\ell j}(x,r)$ for some $j$ if and only if
\begin{equation}\label{res-equal}
\res (\gamma_{x,r}b_{x,r,\ell,i_1})  = \res (\gamma_{x,r}b_{x,r,\ell,i_2}),
\end{equation}
where $\res:\cO_K\to k_K$ is the natural projection.
By cutting $U$ into finitely many pieces again,
we may assume
that the sets $S_{\ell j}:=S_{\ell j}(x,r)$ do not depend on $(x,r)$.
Since $b_{x,r,\ell,1} = 0$, at least for one $\ell$ there are at least two different sets
$S_{\ell,j}, S_{\ell,j'}$.
Define for each $\ell, j$ and for $(x,y)\in U$
$$
f_{\ell j}(x,y) := \sum_{i\in S_{\ell j}} G_{\ell i}(\varphi(x,y)) \psi_K(h_{\ell i}(x,y))
$$
and consider these functions $(f_{\ell j})_{\ell,j}$ as a single family.
The total number of summands of the family $(f_{\ell j})_{\ell,j}$ is the same
as for the functions $f_\ell$, but there are more functions $f_{\ell j}$ than $f_j$, so
we can apply induction on $N$ to this family $(f_{\ell j})_{\ell,j}$, with the extra condition 3) and 4) for $\varphi_1$ as part of the desired properties.
Thus we find an integer $d\geq 0$, a definable surjection $\varphi:U\to V$ over $X$, definable functions $h_{\ell ji}:U\to K$, and functions $G_{\ell ji}$ with properties 1), 2), 3) and 4) for $\varphi_1$ and for this family.

Let us write $U_{x,r}$ for the sets defined by $\varphi$ as in condition 2). Since
$\varphi_1=\theta\circ \varphi$ for some definable $\theta$, one has $U_{x,r}\subset U^1_{x,r'}$ for each $(x,r)$ and $(x,r')= \theta(x,r)$. %Hence, the equalities (\ref{jac1}) and (\ref{jac2}) survive
By cutting $U$ into pieces as before, we may assume that, for each
$x$ and $r$,
not all $h_{\ell i}(x,\cdot)$ are constant modulo $(\varpi_K)$ on $U_{x,r}$, since, as before, this would bring us back to the case $N=0$ for our original family $(f_\ell)_\ell$ via Lemma \ref{avoidSkol}.

We will now show that the subset $M_{x,r}$ of $U_{x,r}$ consisting of those $y$ satisfying both inequalities
\begin{equation}\label{goal-iml}
\sup_{\ell,j,i} |G_{\ell ji}(x,r)|_{\CC}\leq \sup_{\ell,j} |f_{\ell j}(x,y)|_{\CC} \leq \sup_\ell  |f_\ell(x,y)|_{\CC}
\end{equation}
has big volume in the sense that
\begin{equation}\label{vol-goal}
\Vol( U_{x,r}) \le q_K^{d+1}  \Vol(M_{x,r}).
\end{equation}

Once this is proved, we are done for our original family $(f_\ell)_\ell$ by replacing $d$ with $d+1$ while keeping the data of the $\varphi$, $G_{\ell ji}$, and $h_{\ell ji}$.

Thus, to finish the proof, we fix $x$ and $r$ and it remains to
show that $M_{x,r} $ as given by (\ref{goal-iml}) has the property (\ref{vol-goal}).
Consider the partition of the ball $U_{x,r}$ into the balls $B_\xi$ of the form $\xi+\gamma_{x,r}\cO_K$.
(The ball $U_{x,r}$ is indeed a union of such balls $B_\xi$ by our choice of $\gamma_{x,r}$ since there exists a $h_{\ell i}(x,\cdot)$ that is non-constant modulo $(\varpi_K)$
on $U_{x,r}$.) Firstly we will show that $|f_{\ell j}(x,\cdot)|_{\CC}$ is constant on each such
$B_\xi$. Secondly we will show that for each such $B_\xi$ there is a sub-ball $B'_\xi \subset B_\xi$ with $\Vol(B_\xi) = q_K\cdot \Vol(B'_\xi)$ and such that
the second inequality of (\ref{goal-iml}) holds for all $y\in B'_\xi$. These two facts together with the previous application of the induction hypothesis imply (\ref{vol-goal}) and thus finish the proof for $m=1$. Fix $B_\xi\subset U_{x,r}$
and write
$y = \xi + \gamma_{x,r}y' \in B_\xi$ for $y' \in \cO_K$.
 By (\ref{jac1}), (\ref{jac2}), and (\ref{res-equal}), for each $\ell$ and $j$ there is a constant
$c_{\ell j} \in \CC$ such that
\[
f_{\ell j}(x,y) = c_{\ell j} \psi_K(b'_{\ell j}y'),
\]
where we can take $b'_{\ell j} = \gamma_{x,r}b_{x,r',\ell,i}$ for any $i \in S_{\ell j}$ where $r'$ is such that $U_{x,r}\subset U^1_{x,r'}$.
This shows that $|f_{\ell j}(x,\cdot)|_{\CC}$ is constant on
$B_\xi$. We now only have to construct $B'_\xi$.
By renumbering, we can suppose that on $B_\xi$,
$|f_{1,1}|_{\CC}$ is maximal among the $|f_{\ell j}|_{\CC}$,
so that the middle expression of (\ref{goal-iml}) is equal to $|f_{1,1}|_{\CC}$.
In particular, it suffices to choose $B'_\xi$ such that $|f_{1,1}|_{\CC} \le |f_1(x,y)|_{\CC}$ for all $y \in B'_\xi$.
 Now let $\psi$ be the additive character of $\FF_{q_K}$ satisfying
$\psi_K(y') = \psi(\res(y'))$ for $y' \in \cO_K$.
By (\ref{res-equal}), we have $\res(b'_{1j}) \ne \res(b'_{1j'})$ for each $j \ne j'$,
so we can apply Corollary~\ref{linear} to
\[
\tilde{f}:\FF_{q_K} \to \CC : \tilde{y} \mapsto \sum_{j} c_{1j} \psi(\res(b'_{1j})\cdot \tilde{y})
\]
and get an $\tilde{y}_0 \in \FF_{q_K}$ with
$|c_{1,1}|_{\CC} \le |\tilde{f}(\tilde{y}_0)|_{\CC}$.
Set $B'_{\xi} := \{\xi + \gamma_{x,r} y' \mid y' \in\res^{-1}(\tilde{y}_0)\}$.
Since
$f_1(x,y) = \tilde{f}(\res(y'))$ and
$|f_{1,1}|_{\CC} = |c_{1,1}|_{\CC}$, we are done.
\end{proof}

\begin{proof}[Proof of Proposition \ref{intCLpexp} for $m>1$]
We proceed by induction on $m$.
Denote $(y_1,\ldots,y_{m-1})$ by $\hat y$.
Apply the $m=1$ case using $(x,\hat y)$ as parameters and
$y_m$ as the only $y$-variable.
This yields in particular an integer $d_1>0$, a surjection $\varphi_1:U
\to V_1 \subset X \times K^{m-1} \times \ZZ^{t_1}$,
and an expression of each $f_\ell$ as a sum of terms of the form
$G_{1}(\varphi_1(x,y))\psi_K(h_1(x,y))$, where we omit the indices
$\ell, i$ to
simplify notation.

Now apply the induction hypothesis to the collection of functions
$G_{1}$,
this time using $\hat y$ as the $y$-variables, and the variables $(x,
r_1)$ as parameters, where $r_1$ is the variable running over
$\ZZ^{t_1}$.
This yields an integer $d_2$, a surjection $\varphi_2:V_1\to V_2
 \subset X \times \ZZ^{t_1}\times \ZZ^{t_2}$, and an
expression of each $G_1$ as a sum of terms of the form
$G_{2}(\varphi_2(x,\hat y, r_1))\psi_K(h_2(x,\hat y, r_1))$.

Now define $\varphi$ as $\varphi_2\circ \varphi_1$ and $d=d_1 + d_2$.
Then each $f_\ell$ is a sum of terms of the form
$G_{2}(\varphi(x, y))\psi_K(h_1(x,y) + h_2(\phi_1(x, y)))$, so
1) is satisfied and 2) also follows easily.
\end{proof}

\begin{proof}[Proof of Theorems \ref{p2pexp} and \ref{p2'pexp}]
Let $f$ be in $\cCexp(X\times K^m)$ for some definable set $X$ and write $f$ as
$\sum_{i=1}^N G_i(\varphi(x,y)) \psi_K(h_i(x,y))$ as in Proposition \ref{intCLpexp} with $s=1$, $f_1=f$, and $U=X\times K^m$,
so that in particular
the $h_i:X\times K^m\to K$ and $\varphi:U\to V$ are definable functions, and the  $G_i(x,y)$ lie in $\cCexp(V)$. For each $i$ let $H_i$ be the function $G_i\circ\varphi$. Since $K$-valued functions can only depend piecewise trivially on $\ZZ$-variables, one has a natural isomorphism
$$
\cCexp(V) \cong \cCexp(X)\otimes_{\cC(X)} \cC(V)
$$
of $\cC(X)$-algebras,
and thus, the $H_i$ may be assumed to lie in $\cCexp(X)\otimes_{\cC(X)} \cC(X\times K^m)$. To define these tensor products of $\AA_q$-algebras we use the natural inclusions $\cC(X)\subset \cC(V)$, $\cC(X)\subset \cCexp(X)$, and $\cC(X)\subset \cC(X\times K^m)$, which are inclusions of algebras of $\CC$-valued functions.

It is clear that for any $x \in X$, if $x \in \Iva(H_i, X)$ for all $i$, then $x \in  \Iva(f, X)$.
Vice versa, if $f(x, \cdot)$ is identically zero, then $H_i(x,\cdot)$ is zero on each set $W_{x,r}$, and since it is constant on each set $U_{x,r}$, $H_i(x,\cdot)$ is identically zero. Thus we just showed:
\begin{equation}\label{i1}
\Iva(f,X) = \bigcap_i \Iva(H_i,X).
\end{equation}
A similar argument shows
\begin{equation}\label{i2}
\Bdd(f,X) = \bigcap_i \Bdd(H_i,X)\ \mbox { and }\
\Int(f,X) = \bigcap_i \Int(H_i,X),
\end{equation}
where in the case of $\Int(f,X)$, we use the inequality between the volumes of $W_{x,r}$ and $U_{x,r}$
given by Proposition \ref{intCLpexp}.
By Corollary \ref{corHX} with $\cH(X)=\cCexp(X)$, applied to each of the functions $H_i$, we find that each of the sets $\Int(H_i,X)$, $\Bdd(H_i,X)$,
$\Iva(H_i,X)$ is equal to a zero locus of a function in $\cCexp(X)$. By Corollary \ref{interspexp} applied to the intersections from (\ref{i1}) and (\ref{i2}), the proof of Theorem \ref{p2pexp} is finished.

Finally apply Corollary \ref{corHX} with $\cH(X)=\cCexp(X)$ to each of the functions $H_i$ to obtain $g_i$ in $\cCexp(X\times K^m)$ such that $\Int(g_i,X)=X$ and
$H_i(x,y)=g_i(x,y)$ whenever $x$ lies in $\Int(H_i,X)$.
Now take $g=\sum_i g_i\psi_K (h_i)$ as required for Theorem \ref{p2'pexp}.
\end{proof}

For the proof of Theorem \ref{p1pexp} we will apply Theorem 8.6.1 (1) of \cite{CLexp}, which has stringent integrability conditions. These conditions can be satisfied by the last part of the statement of Theorem \ref{p2'pexp}.
\begin{proof}[Proof of Theorem \ref{p1pexp}]
Let  $g$ be given by Theorem \ref{p2'pexp}. Since, by the same theorem, $g$ is a finite sum of terms of the form
$$
f_0\psi_K(f_1)
$$
with $f_1:X\times K^m\to K$ definable and $f_0\in \cC(X\times K^m)$ satisfying $\Int(f_0,X)=X$, the function $g$ falls under the scope of Theorem 8.6.1 (1) of \cite{CLexp}, which yields the desired conclusion.
\end{proof}

\section{Transfer principles for integrability and boundedness}\label{s:mot}

In this section we use motivic functions from \cite{CLoes} and \cite{CLexp} to study loci of integrability and other loci, uniformly in all local non-archimedean fields whose residue fields have large characteristic (including $\FF_q\llp t\rrp$). We give uniform analogues of the results of Section \ref{sec:qfixp}. This finally leads to the main results of the paper: the transfer principles for integrability and boundedness.
Again, we first prove the results without oscillation using the constructible motivic functions from \cite{CLoes}, and subsequently in a setting with oscillation using the motivic exponential functions from \cite{CLexp}. We start by recalling the necessary definitions.

\subsection{Notation}\label{sub:notation}

Let  $\cO$ be a ring of integers of a number field. We will use the first order language of Denef-Pas with
coefficients in $\cO[[t]]$, and denote it by $\LPas$. By definable we will from now on mean $\LPas$-definable, without using other coefficients than those from $\cO[[t]]$. Recall that  $\LPas$ has three sorts: the valued field, the residue field, and the value group. The language $\LPas$ has as symbols the usual logical symbols, the language of rings $(+,-,\cdot,0,1)$ with coefficients from $\cO[[t]]$ for the valued field, another copy of the language of rings for the residue field, the Presburger language $(+,-,\leq,\{\cdot\equiv\cdot\bmod n\}_{n>1},0,1)$ for the value group, the symbol $\ord$ for the valuation map on the nonzero elements of the valued field, and the symbol $\ac$ for an angular component map.
All structures for $\LPas$ that we will consider are triples $(L,k_L,\ZZ)$ with $L$ a complete discretely valued field, $\cO_L$ its valuation ring with residue field $k_L$, and value group identified with $\ZZ$, together with the information of how the symbols of $\LPas$ are interpreted in this triple. To fix the meaning of the symbols of $\LPas$ one fixes a ring homomorphism $\lambda_{\cO,L}:\cO[[t]]\to \cO_L$ respecting $1$ and sending $t$ to a uniformizer $\varpi$ of $\cO_L$. If one fixes such $\lambda_{\cO,L}$ then all the symbols of $\LPas$ have a unique meaning where we require that $\ac:
L \to k_L$ is the unique multiplicative map which extends the projection $\cO_L^\times\to
k_L^\times$ and sends $\varpi$ to $1$; it is given by
$$
\ac:L \to k_L:\begin{cases} x\varpi^{-\ord x}\bmod (\varpi) &\quad \mbox{ if $x\not=0$,}\\
0 &\quad \mbox{ if $x=0$},
  \end{cases}
$$
(the other symbols have their natural meaning).
Note that, by the completeness of $L$, a ring homomorphism $\cO\to\cO_L$
and the choice of the uniformizer $\varpi$ of $\cO_L$ determine a ring homomorphism $\cO[[t]]\to \cO_L$ sending $t$ to $\varpi$.

Let $\AO$ be the collection of non-Archimedean local fields $K$ of characteristic zero with a ring homomorphism $\cO\to K$ and a uniformizer $\varpi_K$ of $\cO_K$.
Let $\BO$ be the collection of all local fields $K$ of positive characteristic with a ring homomorphism $\cO\to K$  and a uniformizer $\varpi_K$ of $\cO_K$.
Let $\CO$ be the union of $\AO$ and $\BO$.
 For an integer $M>0$, denote by $\AO[,M]$, $\BO[,M]$, resp.~$\CO[,M]$ those fields in $\AO$, $\BO$, resp.~$\CO$ that have residue characteristic larger than $M$.
For $K$ in $\CO$, write $\cM_K$ for the maximal ideal of $\cO_K$, $k_K$ for the residue field and $q_K$ for the number of elements of $k_K$. For $x\in \cO_K$, denote by $\overline{x} \in k_K$ the reduction of $x$ modulo $(\varpi_K)$.

For $K\in\CO$, write $\cD_K$ for the collection of additive characters $\psi:K\to\CC^\times$
which are trivial on the maximal ideal
$\cM_K$ and which coincide on $\cO_K$ with the character sending $x\in \cO_K$ to
$$
\exp(\frac{2\pi i}{p} {\rm Tr}_{k_K}(\bar x)),
$$
where ${\rm Tr}_{k_K}$
is the trace of $k_K$ over its prime subfield and $p$ is the characteristic of $k_K$.
Note that there is no restriction in only considering additive characters lying in $\cD_K$, since, in our set-up, all other additive characters on $K$ can appear naturally by using a parameter over the valued field.

For any $K\in\CO$, the measure we put on $K^n\times k_K^m\times \ZZ^r$ is the product measure of the Haar measure on $K^n$ normalized so that $\cO_K^n$ has measure $1$ with the discrete measure (the counting measure) on $k_K^m\times \ZZ^r$. Likewise, we endow $K^n\times k_K^m\times \ZZ^r$ with the product topology of the valuation topology on $K^n$ with the discrete topology on $k_K^m\times \ZZ^r$.

\subsection{The motivic setting}\label{mot:set}

We recall the terminology and notation from \cite{CLoes} and \cite{CLexp}.
\subsubsection{Definable subassignments}\label{def}
For any field $k$ of characteristic zero,
we consider the Laurent series field $k\llp t \rrp $ over $k$ with the uniformizer $t$ and the corresponding angular component map and discrete valuation.

Any $\LPas$-formula $\varphi$ in $m$ free valued field variables, $n$ free residue field variables, and $r$ free value group variables, and any field $k$ of characteristic zero which contains our fixed ring of integers $\cO$ as a subring gives rise to a subset of
$$
k\llp t \rrp ^m\times k^n\times \ZZ^r
$$
consisting of the points satisfying $\varphi$, which can be written
symbolically
as follows:
$$
\{(x,y,z)\in k\llp t \rrp ^m\times k^n\times \ZZ^r
\mid \varphi(x,y,z) \};
$$
this subset is denoted by $\varphi_{k\llp t\rrp}$.

By a definable subassignment we mean the map $X$ which sends $k$ to $X(k) := \varphi_{k\llp t\rrp}$ for some $\LPas$-formula $\varphi$, where $k$ runs over characteristic zero fields which contain $\cO$ as a subring. Denote by $h$ the definable subassignment which sends $k$ to the singleton $\{0\}$, also written as $k\llp t \rrp ^0\times k^0\times \ZZ^0$.
(For readers familiar with the language of model theory, note that two formulas $\varphi$ and $\varphi'$ yield the same definable subassignment iff they are equivalent modulo the theory of Henselian valued fields of characteristic $(0, 0)$ with value group elementary equivalent to $\ZZ$ and whose residue fields
contain $\cO$ as a subring.)

For any definable subassignment $X$, and for nonnegative integers $m,n,r$, write $X[m,n,r]$ for the definable subassignment sending $k$ to
$$
X(k)\times k\llp t \rrp ^m\times k^n\times \ZZ^r.
$$
For example, $h[m,n,r]$ sends $k$ to
$k\llp t \rrp ^m\times k^n\times \ZZ^r$. We will also write $X\times \ZZ^r$ for $X[0,0,r]$.

A point on a definable subassignment $X$ consists of a pair $(x,k)$ with $k$ a characteristic zero field having $\cO$ as a subring and with $x$ an element of $X(k)$. We write $|X|$ for the collection of all points that lie on $X$. (We leave it to the reader to choose whether
to consider $|X|$ as an actual class or to work in a fixed large universe.)

The usual set-theoretic operations make sense for definable subassignments. If $X(k)\subset Y(k)$ for each $k$, then we also call $X$ a definable subassignment of $Y$.

By a definable morphism $f:X\to Y$ between definable subassignments $X$ and $Y$ we mean a definable subassignment $G\subset X\times Y$ such that $G(k)$ is the graph of a function from $X(k)$ to
$Y(k)$ for each $k$ and we call $G$ the graph of $f$. We write $f_k$ for the function from $X(k)$ to $Y(k)$ with the graph $G(k)$.

Write $\Def$ for the category of definable subassignments with definable morphisms as morphisms.

\subsubsection{Definable subassignments and local fields}\label{deflf}

We have seen that behind a definable subassignment $X$ lies an $\LPas$-formula $\varphi$ which describes the sets $X(k)$. Clearly such a formula $\varphi$ corresponding to $X$ is not unique. However, if we fix such a $\varphi$ for a definable subassignment $X$ of $h[m,n,r]$, which we call fixing a representative of $X$, then for each $K\in\CO$, we can consider the subset $\varphi_K$ of $K^m\times k_K^n\times \ZZ^r$ consisting of the points satisfying $\varphi$.
%where $X$ is a definable subassignment of $h[m,n,r]$.
Indeed, all the symbols of $\LPas$ can be interpreted in the three sorts $K,k_K,\ZZ$, where elements of $\cO[[t]]$ are interpreted in $K$ via the ring homomorphism $\cO[[t]]\to K$ coming from the ring homomorphism $\cO\to K$ and sending $t$ to the uniformizer $\varpi_K$.

For all motivic objects in this paper, we will make a link with objects (usually sets and functions) on local fields. In particular, given a definable subassignment $X$, we will often implicitly fix a representative $\varphi$ and write $X_K$ instead of $\varphi_K$ for $K$ in $\CO[,M]$ with sufficiently large $M$.
Although $X_K$ depends on the choice of $\varphi$, we have the following phenomenon:
\begin{quote}
For any two representatives $\varphi$ and $\varphi'$ of a definable subassignment $X$, there exists $M>0$ such that $\varphi_K=\varphi'_K$ for all $K$ in $\CO[,M]$.
\end{quote}

\begin{remark}
The operation of taking representatives in this context is similar to the notion of taking a model over $\ZZ$ of a variety defined over $\QQ$ in the context of algebraic geometry, as one typically does for counting the number of $\FF_q$-rational points.
\end{remark}

Similarly, any definable morphism $f:X\to Y$ between definable subassignments gives rise, up to fixing a formula $\gamma$ corresponding to the graph of $f$, to a function
$$
f_K:X_K\to Y_K,
$$
whose graph is $\gamma_K$ for any $K$ in $\CO[,M]$ with $M$ sufficiently large.
(If the characteristic of the residue field of $K$ is small, then $\gamma_K$ might
not define the graph of a function and one may define $f_K$ as being the zero function in this case, by convention.)

\subsubsection{Constructible motivic functions}\label{cmf}
Recall that $h[0,0,1]$ can be identified with $\ZZ$, since $h[0,0,1](k)=\ZZ$ for all $k$. Let $X$ be in $\Def$, that is, let $X$ be a definable subassignment.  A definable morphism $\alpha: X\to h[0,0,1]$
gives rise to a function $|X|\to\ZZ$ (also denoted by $\alpha$) sending a point $(x,k)$ on $X$ to $\alpha_k(x)$. Likewise, such $\alpha$ gives rise to the function $\LL^\alpha$ from $|X|$ to $\AA$ which sends a point $(x,k)$ on $X$ to $\LL^{\alpha_k(x)}$, and where $\AA$ is as in Section \ref{sec:qunif}.

Following \cite{CLoes}, we define the ring $\cP
(X) $ of constructible Presburger functions on $X$ as the subring of
the ring of functions $|X| \rightarrow \AA$ generated by
\begin{enumerate}
\item all constant functions into $\AA$,
\item all functions $\alpha: |X| \rightarrow \ZZ$ with $\alpha:X\to h[0,0,1]$ a definable morphism,
\item all functions of the form $\LL^{\beta}$ with $\beta:X\to h[0,0,1]$ a
definable morphism.
\end{enumerate}

Note that although $|X|$ is not a set, $\cP(X)$ can be regarded as a set since it has
not too many generators.

For $Y$ a definable subassignment of $X$, write $\11_Y$ for the characteristic function of $Y$, sending a point $(x,k)$ on $X$ to $1$ if it lies on $Y$ and to zero otherwise.

Define the  group
$\cQ(X)$
as the quotient of
the free abelian group
over symbols $[Y]$ with $Y$ a definable subassignment of $X[0,m,0]$ for some $m\geq 0$,  by the following scissor relations.
\begin{itemize}

\item[]
\begin{equation}\label{eq1}
[Y] = [Y']
\end{equation}
if there exists a definable isomorphism $Y\to Y'$ which commutes with the projections $Y\to X$ and $Y'\to X$.

\item[]
\begin{equation}\label{eq2}
[Y_1 \cup Y_2] + [Y_1 \cap Y_2]
= [Y_1] + [Y_2]
\end{equation}
for $Y_1$ and $Y_2$ definable subassignments of a common $X[0,m,0]$ for some $m$.
\end{itemize}
We will still write $[Y]$ for the class of $[Y]$ in $ \cQ  ( X ) $ for $Y\subset X[0,m,0]$. Note that in \cite{CLoes} and \cite{CLexp}, the notation $K_0(\RDef_X)$ is used instead of $\cQ(X)$.
Denote by $\cP^0 (X)$ the subring of $\cP (X)$ generated by the characteristic functions ${\bf 1}_Y$ for all definable
subassignments $Y$ of $X$ and by the constant function $\LL$.
Using the canonical ring  morphism  $\cP^0 (X) \rightarrow \cQ(X)$, sending ${\bf 1}_Y$ to $[Y]$ and $\LL$ to the class of $X[0,1,0]$, we define the ring $\cC(X)$  as
$$
\cP(X)\otimes_{\cP^0(X)}\cQ(X).
$$
Elements of $\cC(X)$ are called constructible motivic functions on $X$.

Let $F$ be a function in $\cC(X[m,0,0])$ for some $m\geq 0$. Using notation from \cite{CLoes} Section 13.2, we say that $F$ is motivically $X$-integrable if and only if its class in $C^m(X[m,0,0]\to X)$ lies in $\mathrm{I}_XC(X[m,0,0]\to X)$, where $X[m,0,0]\to X$ is the projection.
We do not need the notion of motivic integrability for the transfer principles, and we refer to \cite{CLoes} for a more detailed definition. Let us just give an intuitive explanation of motivic $X$-integrability. The condition for $F$ to be motivically $X$-integrable is a strong uniform form of the condition that for all $K$ in $\CO[,M]$ with $M$ sufficiently large, $F_K(x,\cdot)$ is integrable over $K^m$ with respect to the Haar measure, for each $x\in X_K$. This motivic condition is defined, via cell decomposition techniques, in terms of $X$-integrability of functions $G$ in $\cP(X\times \ZZ^m)$, and reduces to summability over $\ZZ^m$, as follows. A function $G$ in $\cP(X\times \ZZ^m)$ is considered $X$-integrable if and only if for each $(x,k)\in|X|$, the family  $$
(G_k(x,z)(q))_{z\in\ZZ^m}
$$
is summable (in the usual sense) for each real $q>1$, where we use evaluation at $\LL=q$ as in (\ref{aq}).
In fact, the existence of $g$ as in (\ref{g3}) of Theorem \ref{p2'mot} gives a precise relation between integrability over local fields with large residue field characteristic and motivic integrability.

\subsubsection{Constructible motivic functions and local fields}\label{cmflf}

Each $f$ in $\cP(X)$, with $X$ a definable subassignment, can be written as a finite sum of terms of the form
$a \LL^\beta \prod_{i=1}^\ell \alpha_i$
with $a\in \AA$, and the  $\alpha_i$ and $\beta$ definable morphisms from $X$ to $h[0,0,1]=\ZZ$. Let us take representatives $\alpha_i'$ and $\beta'$ of the $\alpha_i$ and $\beta$, that is, the $\LPas$-formulas describing the graphs. We have seen in Section \ref{deflf} that, for $K$ in $\CO[,M]$ with $M$ sufficiently large, $\alpha_{iK}'$ and $\beta'_K$ are the graphs of functions from $X_K$ to $\ZZ$ and we have denoted these functions by $\alpha_{iK}$ and $\beta_K$. Now we extend this notation to elements $f$ of $\cP(X)$, where we write $f_K$ for the function sending $x\in X_K$ to
$$
\sum_j a_j(q_K) q_K^{\beta_{jK}(x)} \prod_{i=1}^{\ell_j} \alpha_{ijK}(x),
$$
whenever $f$ equals
$$
\sum_j a_j \LL^{\beta_j} \prod_{i=1}^{\ell_j} \alpha_{ij},
$$
where $a_j\in \AA$, $a_{j}(q_K)$ is the evaluation of $a_j$ at $\LL=q_K$ as in (\ref{aq}), and the $\alpha_{ij}$ and $\beta_j$ are definable morphisms from $X$ to $\ZZ$. In a similar sense as in Section \ref{deflf}, the function $f_K:X_K\to \QQ$ is independent of the choice of the representatives for the $\alpha_{ij}$ and $\beta_i$ whenever $K$ is in $\CO[,M]$ with $M$ sufficiently large.

Likewise, since each $g$ in $\cQ(X)$ can be written as $[Y]-[Z]$ for some definable subassignments $Y\subset X[0,n,0]$ and $Z\subset X[0,n',0]$, by taking representatives, one can consider $Y_K$ and $Z_K$ for $K$ in  $\CO[,M]$ with $M$ sufficiently large, and we denote by $g_K$ the function on $X_K$ sending $x\in X_K$ to
$$
\# Y_x - \# Z_x,
$$
where $Y_x$ is the (finite) set $\{r\in k_K^n\mid (x,r)\in Y_K\}$ of size $\# Y_x$ and likewise for $Z_x$.

Since for $f\in\cP^0(X)$ and $f'$ its image in $\cQ(X)$ one has $f_K=f'_K$ for all $K$ in $\CO[,M]$ with $M$ sufficiently large, one can define for $F$ in $\cC(X)$ and for $K$ in $\CO[,M]$ with $M$ sufficiently large, the function $F_K$ as
$$
F_K:X_K\to\QQ : x\mapsto \sum_{i} a_{iK}(x) b_{iK}(x)
$$
whenever $F=\sum_i a_i\otimes b_i$ with $a_i\in\cP(X)$ and $b_i\in\cQ(X)$.
In our usual sense, this is independent of the choice of representatives when $K$ is in $\CO[,M]$ with $M$ sufficiently large.

\subsubsection{Motivic exponential functions}\label{cemf}

Let $X$ be in $\Def $. We consider the category $\RDefe_{X}$
whose objects are the triples $(Y, \xi, g)$ with $Y$ a definable subassignment of $X[0,n,0]$ for some $n\geq 0$, and $\xi : Y \rightarrow h[0,1,0]$ and $g : Y
\rightarrow h[1,0,0]$ definable morphisms. A morphism $(Y',
\xi', g') \rightarrow (Y, \xi, g)$ in $\RDefe_{X}$ is a
definable morphism $h : Y' \rightarrow Y$ which makes a commutative diagram with the projections to $X$ and such that $\xi' = \xi
\circ h$ and $g' = g \circ h$.

To the category $\RDefe_{X}$ one assigns a ring $\cQexp(X)$ defined as follows. As an abelian group it is the
quotient of the free abelian group over the symbols $[Y,
\xi,g]$ with $(Y, \xi,g)$ in $\RDefe_{X}$ by the
following four relations
\begin{equation}\label{r1}
[Y, \xi,g] = [Y', \xi',g']
\end{equation}
for $(Y, \xi,g)$ isomorphic to $(Y',\xi',g')$,
\begin{equation}\label{r2}
\begin{split}
[Y \cup Y', \xi, g] & + [Y \cap Y',
\xi_{|Y \cap Y'},g_{|Y \cap Y'}]\\ &= [Y, \xi_{|Y},
g_{|Y}] + [Y', \xi_{|Y'}, g_{|Y'}]
\end{split}
\end{equation}
for $Y$ and $Y'$ definable subassignments of some common $X[0,n,0]$ for some $n\geq 0$
and $\xi$, $g$ defined on $Y \cup Y'$,
 \begin{equation}\label{r3}
 [Y,\xi,g+g']=[Y,\xi+\bar g,g']
 \end{equation}
for $g:Y\to h[1,0,0]$ a definable morphism with $\ord(g(y))\geq 0$
for all $y$ in $Y$ and $\bar g$ the reduction of $g$ modulo the maximal ideal,
and
 \begin{equation}\label{r5}
 [Y[0,1,0],\xi+p,g]=0
 \end{equation}
when $p:Y[0,1,0]\to h[0,1,0]$ is the projection and when
$g$ and $\xi$ factorize through the projection
$Y[0,1,0]\to Y$.

\begin{lem}[\cite{CLexp}, Lemma 3.1.1]\label{ring}
We may endow $\cQexp(X)$ with a ring structure by setting
$$
[Y, \xi,g] \cdot  [Y', \xi',g'] = [Y \otimes_X Y',
\xi\circ p_Y + \xi'\circ p_{Y'}, g\circ p_Y + g'\circ p_{Y'}],
$$
where $Y \otimes_X Y'$ is the fiber product of $Y$ and $Y'$, $p_Y$
the projection to $Y$, and $p_{Y'}$ the projection to $Y'$.
\end{lem}
By Lemma 3.1.3 of \cite{CLexp} there is a natural injection of rings $\cQ(X)\rightarrow
\cQexp(X)$ sending $[Y]$ to
$[Y,0,0]$.
Hence, we may define the ring $\cCexp(X)$ of
motivic exponential functions by
\begin{equation}
\cCexp  (X):=\cC(X)\otimes_{\cQ(X)} \cQexp(X).
 \end{equation}

\begin{remark}
Note that in \cite{CLexp}, $\cQ(X)$ is denoted by $K_0(\RDef_{X})$,
$\cQexp(X)$ is denoted by $K_0(\RDef^{\rm exp}_{X})$, and  $\cCexp(X)$ is denoted by $\cC(X)^{\rm exp}$.
\end{remark}

Let $F$ be a motivic exponential function in $\cCexp(X[m,0,0])$ for some $m\geq 0$. Using notation from \cite{CLexp}, we say that $F$ is motivically $X$-integrable if and only if its class in $C^m(X[m,0,0]\to X)^{\rm exp}$ lies in $\mathrm{I}_XC(X[m,0,0]\to X)^{\rm exp}$, where $X[m,0,0]\to X$ is the projection. Again the notion of $X$-integrability essentially boils down to the condition of summability for countable families, similarly as explained at the end of Section \ref{cmf}. Property (\ref{p2'pexpmot3}) of Theorem \ref{p2'mot} gives a precise and new relation between integrability over local fields with large residue field characteristic and motivic integrability.

\subsubsection{Motivic exponential functions and local fields}\label{cemflf}

In this section we explain, following \cite{CLexp}, how to find actual functions $f_{K,\psi}:X_K\to\CC$ for $f\in\cCexp(X)$, $K$ in $\CO[,M]$ with $M$ sufficiently large, $\psi\in \cD_K$, and $X$ a definable subassignment. For $f$ in $\cC(X)$ this is already explained above in Section \ref{cmflf}. Take $G=[Y,\xi,g]$ in $\cQexp(X)$, with $Y\subset X[0,n,0]$, take representatives of $Y$, $\xi$, and $g$, and let $K$ be in $\CO[,M]$ with $M$ sufficiently large, so that $\xi_K$ and $g_K$ are functions from $Y_K$ to $k_K$, resp.~to $K$. Then we define $G_{K,\psi}$ as the function sending $x\in X_K$ to the exponential sum
$$
\sum_{r\in Y_x}\psi(\xi_K(x,r) + g_K(x,r)),
$$
which is well defined since $\psi$ is trivial on $\cM_K$, and since $\xi_K(x,r)$ can be considered as an element of $\cO_K\bmod \cM_K$.
Finally, for $f\in\cCexp(X)$, $K$ in $\CO[,M]$ with $M$ sufficiently large, and $\psi\in \cD_K$, we define $f_{K,\psi}$ by
$$
f_{K,\psi}:X_K\to \CC : x\mapsto \sum_i a_{iK}(x) b_{iK,\psi}(x)
$$
whenever $f=\sum_i a_i\otimes b_i$ with $a_i\in\cC(X)$ and $b_i\in\cQexp(X)$.

We recapitulate how $f_{K,\psi}$ is independent of the choice of representatives for $K$ in $\CO[,M]$ with $M$ sufficiently large. For any two different collections $C_1$ and $C_2$ of representatives of the $\LPas$-formulas that go into the description of $f$, there exists $M'$ such that for all $K$ in $\CO[,M']$ and all $\psi\in\cD_K$, one has that $f_{K,\psi}$ is independent of the choice between $C_1$ and $C_2$.

\subsection{The constructible setting}\label{rr}

We find motivic analogues of our thematic results, namely the analogues of the $p$-adic Theorems \ref{p1p-new}, \ref{p2p}, and \ref{p2'p}.

\begin{thm}[Integration]\label{p1qmot}
Let $f$ be in $\cC(X[m,0,0])$ for some $m\geq 0$ and some definable subassignment $X$. Then there exists $g$ in $\cC(X)$ such that for all $K$ in $\CO$ with large enough residue field characteristic,
$$
g_K(x) = \int_{y\in K^m}f_K(x,y),
$$
whenever $x\in \Int(f_K,X_K)$.
\end{thm}
The special case of the above theorem when $f$ is motivically $X$-integrable follows from \cite{CLoes} and \cite{CLexp}.

\begin{thm}[Correspondences of loci]\label{p2mot}
Let $f$ be in $\cC(X[m,0,0])$ for some definable subassignment $X$ and some $m\geq 0$. Then there exists $h_1,h_2,h_3\in \cC(X)$ such that, for all $K$ in $\CO$ with large enough residue field characteristic, the zero locus of $h_{iK}$ in $X_K$ equals $\Int(f_K,X_K)$, resp. $\Bdd(f_K,X_K)$ resp. $\Iva(f_K,X_K)$, when $i$ is $1$, $2$, or $3$, respectively.
\end{thm}

\begin{thm}[Interpolation]\label{p2'mot}
Let $f$ be in $\cC(X[m,0,0])$ for some $m\geq 0$ and some definable subassignment $X$. Then there exists $g$ in $\cC(X[m,0,0])$ such that the following hold for $K$ in $\CO$ with  large enough residue field characteristic.
\begin{enumerate}
\item $f_K(x,y)=g_K(x,y)$ whenever $x\in \Int(f_K,X_K)$, and for all $y\in K^m$,

\item $\Int(g_K,X_K) = X_K$,

\item\label{g3} $g$ is motivically $X$-integrable.
\end{enumerate}
\end{thm}

\subsection{The exponential setting and transfer principles}\label{int:exp}

The following two theorems constitute the new general transfer principles of this paper, for integrability, local integrability, boundedness, and local boundedness.
\begin{thm}[Transfer principle for integrability]\label{mtrel}
Let $f$ be in $\cCexp(X[m,0,0])$ for some $m\geq 0$ and some definable subassignment $X$.
Then, for all $K\in \CO[,M]$ for some large $M$, the truth of each of the following statements depends only on (the isomorphism class of) the residue field of $K$.
\begin{enumerate}

\item For all $x\in X_K$ and for all $\psi\in \cD_K$, the function $f_{K,\psi}(x,\cdot)$ is integrable over $K^m$, that is, $\Int(X_K, f_{K,\psi}) = X_K$ for all $\psi\in \cD_K$.

\item For all $x\in X_K$ and for all $\psi\in \cD_K$, the function $f_{K,\psi}(x,\cdot)$ is \textbf{locally} integrable on $K^m$.
\end{enumerate}
\end{thm}

\begin{thm}[Transfer principle for boundedness]\label{mtbound}
Let $f$ be in $\cCexp(X[m,0,0])$ for some $m\geq 0$ and some definable subassignment $X$.
Then, for all $K\in \CO[,M]$ for some large $M$, the truth of each of the following statements depends only on (the isomorphism class of) the residue field of $K$.
\begin{enumerate}

\item For all $x\in X_K$ and for all $\psi\in \cD_K$, the function $f_{K,\psi}(x,\cdot)$ is bounded on $K^m$.

\item For all $x\in X_K$ and for all $\psi\in \cD_K$, the function $f_{K,\psi}(x,\cdot)$ is \textbf{locally} bounded on $K^m$.
\end{enumerate}
\end{thm}

The transfer principles will follow from
motivic analogues of our thematic results, which we now state in our final, exponential setting.

\begin{thm}[Integration]\label{p1qmotexp}
Let $f$ be in $\cCexp(X[m,0,0])$ for some $m\geq 0$ and some definable subassignment $X$. Then there exists $g$ in $\cCexp(X)$ such that the following holds for all   $K$ in $\CO$ with large enough residue field characteristic, and for all $\psi\in \cD_K$,
$$
g_{K,\psi}(x) = \int_{y\in K^m}f_{K,\psi}(x,y),
$$
whenever $x\in \Int(f_{K,\psi},X_K)$.
\end{thm}

\begin{thm}[Correspondences of loci]\label{p2pexpmot}
Let $f$ be in $\cCexp(X[m,0,0])$ for some definable subassignment $X$ and some $m\geq 0$.
Then there exist $h_1,h_2,h_3\in \cCexp(X)$ such that, for all $K$ in $\CO$ with large enough residue field characteristic and
for each $\psi\in\cD_K$, the zero locus of $h_{iK,\psi}$ in $X_K$ equals respectively
$
\Int(X_K, f_{K,\psi}),
$
$
\Bdd(X_K, f_{K,\psi}),
$ and  $
\Iva(X_K, f_{K,\psi}),
$
for $i=1$, $2$, or $3$ respectively.
\end{thm}

Theorem \ref{p2pexpmot} implies the following corollary by the same reasoning as for Corollary \ref{p2plocexp}. The analogue of Corollary  \ref{p2ploc} (i.e., the statement without the exponentials) in the motivic context holds similarly but is left to the reader.
\begin{cor}\label{p2pmotlocexp}
Let $f$ be in $\cCexp(X[m,0,0])$ for some definable subassignment $X$ and some $m\geq 0$.
Then there exist functions $h_1$ and $h_2$ in $\cCexp(X)$ such that, for all $K$ in $\CO$ with large enough residue field characteristic and
for each $\psi\in\cD_K$, the zero locus of $h_{1K,\psi}$ in $X_K$ equals
$$
\{x\in X_K\mid f_{K,\psi}(x,\cdot) \mbox{ is locally integrable on }K^m\},
$$
and the zero locus of $h_{2K,\psi}$ in $X_K$ equals
$$
\{x\in X_K\mid f_{K,\psi}(x,\cdot) \mbox{ is locally bounded on }K^m\}.
$$
\end{cor}

\begin{thm}[Interpolation]\label{p2'pexpmot}
Let $f$ be in $\cCexp(X[m,0,0])$ for some definable subassignment $X$  and some $m\geq 0$. Then there exist $g$ in $\cCexp(X[m,0,0])$ and $M>0$ such that for all $K$ in $\CO[,M]$ and all $\psi\in\cD_K$ one has
\begin{enumerate}
\item $f_{K,\psi}(x,y)=g_{K,\psi}(x,y)$ whenever  $x$ lies in $\Int(X_K, f_{K,\psi})$,

\item $\Int(X_K, g_{K,\psi}) = X_K$,

\item\label{p2'pexpmot3} $g$ is motivically $X$-integrable.
\end{enumerate}
\end{thm}

\begin{remark}\label{remgeneral}
By standard techniques of motivic integration,
all the above results (\ref{mtrel} to \ref{p2'pexpmot}, and similarly
in the previous sections) imply the corresponding results where
$X[m,0,0]$ is replaced by an arbitrary subassignment $U \subset X[m,0,0]$.
\end{remark}

\subsection{Proofs of the motivic results}
We begin with some preliminaries.
\begin{defn}\label{def:par}
Consider a definable subassignment $X$. A residual parameterization of $X$ is by definition a definable isomorphism over $X$ of the form
$$
\sigma:X\to X_{par}\subset X[0,m,0]
$$
for some $m\geq 0$.
For $F:X\to Y$ a definable morphism, write $F_{par}$ for the corresponding definable morphism $F\circ \sigma^{-1}:X_{par}\to Y$. Likewise, given $f\in \cCexp(X)$, write $f_{par}$ for the natural corresponding function in $\cCexp(X_{par})$, and so on.
\end{defn}

Note that parameterizing is a way of working piecewise in a uniform way (creating, for each $K$, at most $\# k_K^s$ pieces for a parameterization $\sigma$ which introduces $s$ new residue field variables, the pieces being the fibers in of the coordinate projection to $k_K^s$).
In \cite{CLexp}, for $\sigma:X\to X_{par}$ and $f\in \cCexp(X)$, one denotes $f_{par}$ by $(\sigma^{-1})^*(f)$, which is the compositional pull-back of $f$ along $\sigma^{-1}$ and which becomes the actual composition $f_K\circ\sigma^{-1}_K$ when specializing for $K$ in $\CO[,M]$.
Up to residual parameterization (i.e., up to replacing $X$ by $X_{par}$ for a well-chosen residual parameterization),
all results of Section \ref{prelimp} go through in a uniform way, as we now explain.

\begin{defn}[Uniform cells]\label{def::cellmot}
Consider $A \subset  \Lambda[1,0,0]$ for some definable subassignment $\Lambda$. Then $A$ is called a uniform $1$-cell, resp.~a uniform $0$-cell, over $\Lambda$ if there exists $M>0$
such that for all $K$ in $\AO[,M]$ one has that
$A_K$ is a $p$-adic $1$-cell, resp.~a $p$-adic $0$-cell, over $\Lambda_K$.
\end{defn}

The next theorem follows from Denef-Pas cell decomposition \cite{Pas} and the results of \cite{CLip}, Section 6.
\begin{thm}[Uniform version of the cell decomposition Theorem~\ref{cd} and the Jacobian property Proposition~\ref{jacprop}]\label{cdmot}
Consider $X \subset  \Lambda[1,0,0]$ with $\Lambda$ and $X$ definable subassignments and let $f_j:X\to h[0,0,1]$ and $F:X\to h[1,0,0]$ be definable morphisms and let $g_j$ be in $\cC(X)$.
Then there exists a residual parameterization $\sigma:X\to X_{par}\subset X[0,m,0]$ for some $m\geq 0$, such that the following holds.
There exist $M>0$ and a finite partition of $X_{par}$ into definable subassignments $A$ such that for all $K$ in $\AO[,M]$, one has that the sets $A_K$ and the restrictions of the functions $f_{j,par,K}$ and $g_{j,par,K}$, resp.~$F_{par,K}$,
to $A_K$ are as in Theorem \ref{cd}, resp.~Proposition~\ref{jacprop}, with $Y =\Lambda[0,m,0]_K$.
\end{thm}

From now on we will work and prove results for all $K$ in $\CO[,M]$ for some $M$, instead of only in $\AO[,M]$, which will be allowed by the uniform nature of the above results and by the classical Ax-Kochen principle of \cite{AK2} for first order statements in the language $\LPas$. %(This form of the classical Ax-Kochen principle is a direct consequence of the quantifier elimination result from \cite{Pas}.)
This variant of the classical Ax-Kochen principle follows directly from the quantifier elimination result of \cite{Pas}; we will use this variant, for first order statements in the language $\LPas$, as an ingredient in our proofs. For a recent geometric treatment of classical Ax-Kochen principles, see \cite{DenefCol}.

We give uniform variants of some Presburger results. %These results do not involve residual parameterizations.
\begin{prop}[Uniform Rectilinearization]\label{paramrectimot}
Let $Y$ and $X\subset Y\times \ZZ^m$ be definable subassignments. Then there exist finitely many
definable subassignments $A_i\subset Y\times \ZZ^{m}$ and $B_i\subset Y\times \ZZ^{m}$
and definable isomorphisms $\rho_i:A_i\to B_i$ over $Y$ such that the following holds for large enough $M$ and each $K$ in $\CO[,M]$.
The sets $A_{i,K}$ are disjoint and their union equals $X_K$, and
for every $i$, the function $\rho_{i,K}$ is linear over $Y_K$ and for each $y\in Y_K$, the set
$B_{i,K,y}$ is
a set of the form
$\Lambda_y\times \NN^{\ell_i}$ for a finite subset $\Lambda_y\subset \NN^{m-\ell_i}$ depending on $y$, with
%and for
an integer $\ell_i\geq 0$ only depending on
$i$.
\end{prop}
\begin{proof}
The %goes by exactly the same proof
proof goes exactly the same way as the one for Proposition \ref{paramrectip}, using the quantifier elimination result for valued field variables of \cite{Pas} instead of the one for fixed $K$. (As an alternative proof one may again adapt the proof of Theorem 3 of \cite{CPres}.)
\end{proof}

The following uniform variants of Corollaries \ref{p2pZ} and \ref{p2'pZ} will be used to prove Theorems \ref{p2mot} and \ref{p2'mot}.
\begin{cor}\label{p2motZ}
Let $f$ be in $\cC(X\times \ZZ^m)$ for some definable subassignment $X$ and some $m\geq 0$. Then there exist
$h_1,h_2$ and $h_3$ in $\cC(X)$ such that for large enough $M$ and each $K$ in $\CO[,M]$
\begin{equation}\label{p2eqmotZ}
\Int(f_{K},X_{K}) = Z(h_{1,K}),
\end{equation}
\begin{equation}\label{p2bmotZ}
\Bdd(f_{K},X_{K}) = Z(h_{2,K}),
\end{equation}
and
\begin{equation}\label{p3motZ}
\Iva(f_{K},X_{K}) = Z(h_{3,K}).
\end{equation}
\end{cor}
\begin{cor}\label{p2'motZ}
Let $f$ be in $\cC(X\times \ZZ^m)$ for some definable subassignment $X$ and some $m\geq 0$.
Then there exists $g$ in $\cC(X\times \ZZ^m)$ with $\Int(g,X)=X$ and such that for large enough $M$ and each $K$ in $\CO[,M]$ one has
$f_K(x,y)=g_K(x,y)$ whenever $x$ lies in $\Int(f_K,X_K)$. Moreover, one can take $g$ motivically $X$-integrable.
\end{cor}
\begin{proof}[Proofs of Corollaries \ref{p2motZ} and \ref{p2'motZ}]
The same proofs as those of Corollaries \ref{p2pZ} and \ref{p2'pZ} apply, where one uses Theorem \ref{paramrectimot} instead of Theorem \ref{paramrectip}.
We us give some extra details how one can ensure that $g$ can be taken motivically $X$-integrable. If $X\subset h[0,n,r]$ for some $n$ and some $r$, then the proofs of Corollaries \ref{p2pZ} and \ref{p2'pZ} directly adapt by the definition of $\cC(h[0,n,r])$ and the orthogonality between the value group and residue field in $\LPas$. (This orthogonality is the property that any definable subassignment of $h[0,n,r]]$ equals a finite union of Cartesian products of definable subsets of $\ZZ^r$ and of $h[0,n,0]$ which follows from quantifier elimination of valued field variables in the language $\LPas$, see \cite{Pas}.) The general case follows from the observation, which also follows from quantifier elimination of valued field variables in $\LPas$, that there exists a definable morphism
$$
\Delta:X[0,0,m]\to h[0,n,r+m]
$$
over $\ZZ^m$, and a motivic constructible function $H$ in $\cC(h[0,n,r+m])$, such that the compositional pull-back $\Delta^*(H)$ (with notation from \cite{CLoes}) equals $f$. The just obtained special case that $X\subset h[0,n,r]$, applied to $H$, yields a motivically $h[0,n,r]$-integrable function $g_H$ in $\cC(h[0,n,r+m])$. This allows us to put $g=\Delta^*(g_H)$ as required by Corollary \ref{p2'motZ}.
\end{proof}

The uniform analogue of Lemma \ref{avoidSkol} goes as follows.

\begin{lem}\label{avoidSkolmot}
Let $A^0\subset Y^0[1,0,0]$ be a definable subassignment, and let $h^0:A^0\to h[1,0,0]$ be a definable morphism for some definable subassignment $Y^0$. Let $M_0\geq 0$ be given.  Suppose that for each $K\in \CO[,M_0]$, for each $y\in Y^0_K$, and for each maximal ball $B$ contained in $A^0_{K,y}$, $h^0_K(y,\cdot)$ is constant modulo $(\varpi_K)$ on $B$.
Then there exists a residual parameterization $\sigma:A^0\to A^0_{par}\subset A^0[0,s,0]$ for some $s\geq 0$, such that, if we write $A$ for $A^0_{par}$, $Y$ for $Y^0[0,s,0]$  and $h$ for $h^0\circ\sigma^{-1}$, the following holds. There exist $M\geq M_0$, a finite partition of $A$ into definable subassignments $A_j$ and definable morphisms $h_{j}:Y\to h[1,0,0]$ such that for each $K$ in $\CO[,M]$
$$
|h_K(y,t)-h_{j,K}(y)| \leq 1
$$
holds for each $(y,t)\in A_{j,K}$ and for each $j$.
\end{lem}
\begin{proof}
This follows from Theorem 2.2.2 of \cite{CLexp}, or, alternatively, from the same proof as the one of Lemma \ref{avoidSkol}, where automatically all the $m_j$ are equal to $1$ at the end of that proof.
\end{proof}

Up to using a residual parameterization, the uniform version of Proposition \ref{intCLpexp} also holds, as follows.

\begin{prop}\label{intCLpexpmot}
Let $X^0$ and $U^0 \subset X^0[m,0,0]$ be definable subassignments and
let $f^0_1, \dots, f^0_s$
be in $\cCexp(U^0)$.
There exists a residual parameterization $\sigma:U^0\to U^0_{par}\subset U^0[0,s,0]$ for some $s\geq 0$, such that, if we write $X$ for $X^0[0,s,0]$, $U$ for $U^0_{par}$ and $f_{\ell}$ for $f_{\ell}^0\circ \sigma^{-1}$, then the following holds.
There exist an integer $d\geq 0$, definable morphisms $h_{\ell i}:U\to h[1,0,0]$, a definable surjection $\varphi:U\to V\subset X[0,0,t]$ for some $t\geq 0$, and functions $G_{\ell i}$ in $\cCexp(V)$ such that for each $K$ in $\CO$ with large enough residue field characteristic and for each $\psi$ in $\cD_K$, conditions 1) and 2) of Proposition \ref{intCLpexp} hold for $d$, $f_{\ell K,\psi}$, $U_K$, $V_K$, $\varphi_K$, $h_{\ell i K}$, and $G_{\ell i K,\psi }$.
\end{prop}
\begin{proof}
The proof of Proposition \ref{intCLpexp} works uniformly in $K$ in $\AO[,M]$ for large enough $M$, where one uses Lemma \ref{avoidSkolmot} and relation (\ref{r3}) instead of Lemma \ref{avoidSkol}, and Theorem \ref{cdmot} instead of both Theorem \ref{cd} and
Proposition~\ref{jacprop}. The only difference in the uniform proof is that we have to use residual parameterizations.
Each time we need a parameterization, for the remainder of the proof we replace all objects (sets and functions) by the corresponding parameterized objects, similarly as in the statement of the proposition. This happens when we apply Theorem \ref{cdmot}, Theorem \ref{avoidSkolmot}, Lemma \ref{avoidSkolmot}, or the induction hypothesis.
\end{proof}

\begin{lem}\label{conjumot}
Let $X$ be a definable subassignment and
let $f$ be in $\cCexp(X)$. Then there exists a function $g$ in $\cCexp(X)$ such that for each $K$ in $\CO[,M]$, each $\psi$ in $\cD_K$, and each $x\in X_K$,
$f_{K,\psi}(x)$ and $g_{K,\psi}(x)$ are conjugate complex numbers.
%In particular, $f_{K,\psi}(x)g_{K,\psi}(x)$ equals the complex norm of $f_{K,\psi}(x)$ for each $x\in X_K$, and thus, $Z(f_{K,\psi}) = Z(f_{K,\psi}g_{K,\psi})$.
\end{lem}
\begin{proof}
For each $K$ in $\CO[,M]$ for large enough $M$ the function $g$ can be obtained from $f$ by putting a minus sign in each of the arguments of the additive characters which occur in $f$. More precisely, write $f$ as a finite sum of terms of the form $g_i\otimes [Y_i,\xi_i,h_i]$ for $g_i$ in $\cC(X)$, $Y_i$ a definable subassignment of $X[0,n_i,0]$ for some $n_i$ and definable functions $\xi_i,h_i$ on $Y_i$, and define $g$ as the corresponding sum with terms $g_i\otimes [Y_i,-\xi_i,-h_i]$. Since the $g_{iK}$ take real values, $g$ is as desired.
\end{proof}
The analogue of Corollary \ref{interspexp} also holds.
\begin{cor}\label{interspexmot}
Let $X$ be a definable subassignment and let $h_i$ be in $\cCexp(X)$ for $i=1,\ldots,N$.
Then there exist $f$ and $g$ in $\cCexp(X)$ and $M>0$ such that for each $K$ in $\CO[,M]$ and each $\psi$ in $\cD_K$,
$$
Z(f_{K,\psi}) = \bigcap_{i=1}^{N} Z(h_{i,K,\psi})\ \mbox{ and }\ Z(g_{K,\psi}) = \bigcup_{i=1}^{N} Z(h_{i,K,\psi}).
$$
The corresponding statement for $h_i$ in $\cC(X)$ also holds, yielding $f,g\in\cC(X)$.
\end{cor}
\begin{proof}
For $f$ one can take $\sum_{i=1}^N h_i \overline{h_i}$, where $\overline{h_i}$ is the complex conjugate of $h_i$ as given by Lemma \ref{conju}. For $g$ one simply takes the product of the $h_i$.
\end{proof}

The final ingredients are the following basic forms of the Transfer Principle and the Specialization Principle of \cite{CLexp}.
The transfer principle enables one to change the characteristic of $K$ for equalities between the specializations of exponential motivic functions.
The specialization principle of \cite{CLexp} states that taking motivic integrals commutes with taking specializations to local fields and taking the classical integral over the local field.

\begin{prop}[\cite{CLexp}, Proposition 9.2.1]\label{strong} Let $\varphi$ be in
$\cCexp (X)$ for some definable subassignment $X$. Then, there exists an integer $M$
such that for all $K_1,K_2$ in $\CO[,M]$ with isomorphic residue fields $k_{K_1}$,
$k_{K_2}$, the following holds:
\begin{gather*}
\varphi_{K_1,\psi_{K_1}} = 0 \mbox{ for all }
\psi_{K_1}\in\cD_{K_1}
\\
\text{if and only if}\\
 \varphi_{K_2,\psi_{K_2}} = 0\mbox{ for all  }
\psi_{K_2}\in\cD_{K_2}.
\end{gather*}
\end{prop}

\begin{prop}[\cite{CLexp}, Theorem 9.1.4]\label{SpecialP}
Let $\varphi$ be in $\cCexp (X\subset \Lambda[m,0,0])$ for some definable subassignments $X$, $\Lambda$ and some $m\geq 0$.  Suppose that $\varphi$ is motivically $\Lambda$-integrable, namely, with notation from \cite{CLexp}, the class ${\overline \varphi}$ of $\varphi$ in  $C^m(X\to \Lambda)^{\rm exp}$ lies in ${\rm I}C(X\to \Lambda)^{\rm exp}$.
Let $\theta\in \cCexp (\Lambda)$ be the motivic integral of ${\overline \varphi}$ relative to the projection $X\to\Lambda$, namely, $\theta = \mu_\Lambda({\overline \varphi})$ with notation from (8.7.10) of \cite{CLexp}.
Then, there exists an integer $M$
such that for all $K$ in $\CO[,M]$ and all $\psi$ in $\cD_K$ the following holds
for each $\lambda\in\Lambda_K$. The function $w\mapsto \varphi_{K,\psi}(\lambda,w)$ is integrable over $X_{K,\lambda}$ against the Haar measure on $K^m$, and,
\begin{gather*}
\int_{w\in X_{K,\lambda}}  \varphi_{K,\psi}(\lambda, w) |dw|  = \theta_{K,\psi}(\lambda),
\end{gather*}
with $|dw|$ the normalized Haar measure on $K^m$.
\end{prop}

With all this at hand, we are ready to complete the proofs of our main results.

\begin{proof}[Proof of Theorem \ref{p2mot}]
Let $X$ and $f\in\cC(X[m,0,0])$ be given as in the theorem.
The proof goes in two steps. First we show that there is a suitable residual parameterization $\sigma$ of $X[m,0,0]$ such that the theorem holds for the function $f_{par}$ in $\cC(X[m,0,0]_{par})$, where,
as in Definition \ref{def:par}, $X[m,0,0]_{par}$ and $f_{par}$ are obtained from $X[m,0,0]$ and $f$ using $\sigma$. In the second step, we show that the result for $X[m,0,0]_{par}$ and $f_{par}$ implies the result for $X[m,0,0]$ and $f$ itself.

Let us now treat the first step. We follow the steps of the proof of Theorem \ref{p2p}. By an inductive application of the Cell Decomposition Theorem \ref{cdmot}, we find a residual parameterization
$$
\sigma:X[m,0,0]\to X[m,0,0]_{par} \subset X[m,s,0]
$$
of $X[m,0,0]$ for some $s$, a finite partition of $X[m,0,0]_{par}$ into definable subassignments $A_i$, and an $M>0$, such that for each $K\in \CO[,M]$ the following holds, where we write $X'$
for $X[0,s,0]$.
The nonempty sets among the $A_{i,K}$ are full cells over $X'_K$ which together form a finite partition of $(X[m,0,0]_{par})_K$, and, for nonempty $A_{i,K}$, the restriction $(f_{par,K})|_{A_{i,K}}$ factors through the skeleton map $s_{A_{i,K}/X'_K}$ of $A_{i,K}$ over $X'_K$. As in the proof of Theorem \ref{p2p}, we identify each skeleton with a definable set, for example by replacing $\{+\infty\}$ by a disjoint copy of the singleton $\{0\}$.
Let us write $f_{i,K}$ for the map from the skeleton $S_{/X'_K}(A_{i,K})$ of $A_{i,K}$ over $X'_K$ to $\AA_{q_K}$ induced by $f_{par,K|A_{i,K}}$. Then the function $f_{i,K}$ lies in $\cP_{q_K}(S_{/X'_K}(A_{i,K}))$ for each $i$.
The function $M_{0,K}$ sending $z\in S_{/X'_K}(A_{i,K})$ to the volume of the fiber $(s_{A_{i,K}/X'_K})^{-1}(z)$, taken inside $K^m$, lies in $\cC(S_{/X'_K}(A_{i,K}))$ and clearly is the specialization of a motivic constructible function which is the product of a characteristic function of a definable subassignment and a motivic constructible function of the form $\LL^\alpha$ for some definable morphism $\alpha$ (see Remark \ref{skel}).
Hence, also $z\mapsto f_{i,K}(z)\cdot M_{0,K}(z)$ lies in $\cC( S_{/X'_K}(A_{i,K}) )$ and is the specialization of a motivic constructible function $\tilde f_i$.  For each $i$, let $X'_i$ be the image of the subassignment $A_{i}$ under the projection to the subassignment $X'$.
By Corollary \ref{p2motZ}, we now obtain for each $i$ that there are $h_{i1}$, $h_{i2}$ and $h_{i3}$ in
$\cC(X'_i)$ such that
$$
\Int(\tilde f_{i,K},X'_{i,K}) = Z(h_{i1,K}), %\{x\in X'_{i,K}\mid h_{i1,K}(x)=0\},
$$
$$
\Bdd(\tilde f_{i,K},X'_{i,K})= Z(h_{i2,K}), %\{x\in X'_{i,K}\mid h_{i2,K}(x)=0\},
$$
and
$$
\Iva(\tilde f_{i,K},X'_{i,K}) = Z(h_{i3,K}). %\{x\in X'_{i,K}\mid h_{i3,K}(x)=0\}.
$$
Extend each of the $h_{ij}$ by zero to a function ${\tilde h}_{ij}$ in $\cC(X')$. Now for $j=1,2,$ or $3$, let $h_j$ be the function
$$
\sum_{i}{\tilde h}_{ij}^2.
$$
Then the $h_j\in \cC(X')$ for $j=1,2,3$ are as required by the theorem for $f_{par}$ and $X[m,0,0]_{par}$ instead of for $f$ and $X[m,0,0]$. This finishes the first step.

In the second and final step we show how to get rid of the residue field parameters that were introduced with $\sigma$. Let us denote by $h_{j,par}$ the functions obtained from step one, which are as desired by the theorem but for $f_{par}$ and $X[m,0,0]_{par}$ instead of for $f$ and $X[m,0,0]$, with residual parameterization $\sigma$. %Let us denote the inverse of $g$ by $v$. Note that $v$ is in fact induced by a coordinate projection of the form $X[0,t,0]\to X$ for some $t\geq 0$, by the definition of residual parameterizations.
Let us write $v:X'=X[0,s,0]\to X$ for the natural projection. Observe that $v$ is a coordinate projection which only omits some residue field variables.
Now we define $h_j$ as the motivic integral of $h_{j,par}^2$ relative to $v:X' \to X$, namely, $\mu_{X}(h_{j,par}^2)$ with notation from Section 14.2 of \cite{CLoes} (which coincides with notation from (8.7.10) of \cite{CLexp}, and which in this case coincides with $v_!(h_{j,par}^2)$ of Section~5.6 of \cite{CLoes}). By Proposition \ref{SpecialP}, $h_{j,K}(x)$ for $x\in X_K$ equals the sum of the terms $h_{j,par,K}^2(x,\xi)$,
where $\xi$ runs over the fiber $v^{-1}_K(x) \subset X'_K$, which is a finite set of size at most $q_K^s$. Since a sum of squares of real numbers is nonzero if and only if each of the real numbers is nonzero, the $h_i$ are as desired. This completes the proof of Theorem \ref{p2mot}.
\end{proof}

\begin{proof}[Proof of Theorem \ref{p2'mot}]
Let $X$ and $f\in\cC(C[m,0,0])$ be given as in the theorem.
The proof goes again in two steps, similar to the proof of Theorem \ref{p2mot}. First we show that there is a suitable residual parameterization $\sigma$ of $X[m,0,0]$ such that the theorem holds for the function $f_{par}$ in $\cC(X[m,0,0]_{par})$, where $X[m,0,0]_{par}$ and $f_{par}$. In the second step, we again show that the result for $X[m,0,0]_{par}$ and $f_{par}$ implies the result for $X[m,0,0]$ and $f$ itself.
For the first step, one repeats the first step of the proof of Theorem \ref{p2mot} up to the introduction of the $X'_i$; we use the same notation as in the proof of Theorem \ref{p2mot}. One then proceeds as follows. By Corollary \ref{p2'motZ}, one finds for each $i$ a function $g_i$
in $\cC(S_{/X}(A_i))$ such that $g_i$ is motivically $X_i'$-integrable and,  for all $K$ in $\CO[,M]$ for a large $M$, $\Int(g_{i,K},X_{i,K}')=X_{i}'$ and $\tilde f_{i,K}(x,s)=g_{i,K}(x,s)$ whenever $x$ lies in $\Int(\tilde f_{i,K},X'_{i,K})$.  Now for $g$ one takes the function in $\cC(X[m,0,0]_{par})$ which equals $f_{par}$ on the definable subassignment where $M_0=0$ and which equals  $\frac{g_{i}}{M_0}$ elsewhere. Then this $g$ is as desired by Theorem \ref{p2'mot} for $f_{par}$ and $ X[m,0,0]_{par}$
instead of for $f$ and $X[m,0,0]$. (See Remark \ref{skel} to see that dividing by $M_0$ in this way is harmless.) This finishes the first step.

For the second step, let us denote by $g_{par}$ the function obtained from step one, which is as desired by the theorem but for $f_{par}$ and $X[m,0,0]_{par}$ instead of for $f$ and $X[m,0,0]$, and with residual parameterization $\sigma$. Since $\sigma$ is a definable isomorphism which is moreover a coordinate projection which omits only residue field variables, we can define $g$ as $\sigma^*(g_{par})$, being nothing else than the pull-back of $g_{par}$ in the notation of \cite{CLoes}. For $K$ in $\CO[,M]$ for large $M$, $g_K$ equals the composition $g_{par,K}\circ \sigma_K$, by the definition of pull-back. Hence, the integrability conditions for each $K$ in $\CO[,M]$ for large $M$ are preserved. By construction of the motivic integrability conditions, also motivic integrability is preserved when passing from $g_{par}$ to $g$, in the form desired by the theorem. Recall that $f_{par,K}$ equals $f_K\circ \sigma_K^{-1}$, by the definition of $f_{par}$. Hence, the equality between the $g_{par,K}$ and $f_{par,K}$
on the integrable locus of $f_{par,K}$ yields the desired equality between the $g_K$ and the $f_K$ on the integrable locus of $f_K$. Hence, $g$ is as desired by the theorem for $f$.
\end{proof}

Similarly to the $p$-adic case, the proofs of Theorems \ref{p2mot} and \ref{p2'mot} yield the following slightly more general variant.

\begin{cor}\label{corHXmot}
Let $f$ be in $\cCexp(X)\otimes_{\cC(X)} \cC(X[m,0,0])$ for some definable subassignment $X$ and some $m\geq 0$.
Then there exist  $M\geq 0$, $h_1,h_2$, $h_3$ in $\cCexp(X)$ and $g$ in $\cCexp(X)\otimes_{\cC(X)} \cC(X[m,0,0])$ such that for each $K$ in $\CO[,M]$ and each $\psi$ in $\cD_K$, the zero loci of the $h_{i,K,\psi}$ equal respectively
\begin{equation*}%\label{p2eqp}
\Int(f_{K,\psi},X_K),\qquad
\Bdd(f_{K,\psi},X_K),\qquad\text{and}\qquad
\Iva(f_{K,\psi},X_K),
\end{equation*}
and such that $\Int(g_{K,\psi},X_K)=X_K$ and
$f_{K,\psi}(x,y)=g_{K,\psi}(x,y)$ whenever $x$ lies in $\Int(f_{K,\psi},X_K)$. Moreover, any such $g$ can be written as a finite sum of terms of the form
$$
h_i\cdot f_i
$$
with $h_i\in \cCexp(X)$ and $f_i\in \cC(X[m,0,0])$ satisfying $\Int(f_{i,K,\psi},X_K)=X_K$, and such that the $f_i$ are motivically $X$-integrable.
\end{cor}

\begin{proof}[Proof of Theorems \ref{p2pexpmot} and \ref{p2'pexpmot}]
The proof consists of the two usual steps, as in the proofs of Theorems \ref{p2mot} and \ref{p2'mot}. In the first step, we prove the result up to a residual parameterization $\sigma$. For that step, let us change the notation: denote the given function
$f\in \cCexp(X[m,0,0])$ by $f_0$,
%let us denote by $f_0$ the given motivic function in $\cCexp(X[m,0,0])$, instead of by $f$,
and let us write $X^0$ for $X$ and $U^0$ for $X^0[m,0,0]$. By Proposition \ref{intCLpexpmot}, there exists a residual parameterization $\sigma:U^0\to U^0_{par}\subset X^0[0,s,0]$ such that, if we write $X$ for $X^0[0,s,0]$, $U$ for $U^0_{par}$ and $f$ for $f^0\circ \sigma^{-1}$, then the following holds. There exist an integer $d\geq 0$, definable morphisms $h_{  i}:U\to h[1,0,0]$, a definable surjection $\varphi:U\to V\subset X[0,0,t]$ for some $t\geq 0$, and functions $G_{  i}$ in $\cCexp(V)$ such that for each $K$ in $\CO$ with large enough residue field characteristic and for each $\psi$ in $\cD_K$, conditions 1) and 2) of Proposition \ref{intCLpexp} hold for $d$, $f_{  K,\psi}$, $U_K$, $V_K$, $\varphi_K$, $h_{  i K}$, and $G_{  i K,\psi }$.
For each $i$ let $H_i$ be the function $G_i\circ\varphi$. By (3.3.1) and Proposition 3.3.2 of \cite{CLexp}, one has a natural isomorphism
$$
\cCexp(V) \cong \cCexp(X)\otimes_{\cC(X)} \cC(V)
$$
of $\cC(X)$-algebras,
and thus, the $H_i$ may be considered to lie in $\cCexp(X)\otimes_{\cC(X)} \cC(X[m,0,0])$. For these tensor products of $\cC(X)$-algebras we use the natural inclusions of $\cC(X)$-algebras $\cC(X)\subset \cC(V)$, $\cC(X)\subset \cCexp(X)$, and $\cC(X)\subset \cC(X[m,0,0])$.
From having 2) of Proposition \ref{intCLpexp} for the mentioned data, it follows that for $K$ in $\CO[,M]$, any $\psi$ in $\cD_K$ and for any $x \in X_K$, if $x \in \Iva(H_{i,K,\psi}, X_K)$ for all $i$, then $x \in  \Iva(f_K, X_K)$.
Vice versa, if $f_K(x, \cdot)$ is identically zero, then $H_{i,K,\psi}(x,\cdot)$ is zero on each set $W_{K,x,r}$, and since the function $H_{i,K,\psi}(x,\cdot)$ is constant on each set $U_{K,x,r}$, it is identically zero. Thus we just showed:
$$
\Iva(f_{K,\psi},X_K) = \bigcap_i \Iva(H_{i,K,\psi},X_K).
$$
A similar argument shows
$$
\Bdd(f_{K,\psi},X_K) = \bigcap_i \Bdd(H_{i,K,\psi},X_K),
$$
$$
\Int(f_{K,\psi},X_K) = \bigcap_i \Int(H_{i,K,\psi},X_K),
$$
where in the case of $\Int(f_{K,\psi},X_K)$, we use the inequality between the volumes of $W_{K,x,r}$ and $U_{K,x,r}$
from 2) of Proposition \ref{intCLpexp}.
 Now apply Corollary \ref{corHXmot} to each of the functions $H_i$ to find $h_{ij}$ in $\cCexp(X)$ for $j=1,2,3$ and $g_i$ in $\cCexp(X\times K^m)$.
By Corollary \ref{interspexmot} for the finite intersections displayed above, Theorem \ref{p2pexpmot} holds for $f_{par}$.
For $g$ as required by theorem \ref{p2'pexpmot} for $f_{par}$ one takes the obvious function in $\cCexp(X[m,0,0])$ which specializes to $\sum_i g_{i,K}\psi  (h_{i,K})$ for $K$ in $\CO[,M]$ and $\psi$ in $\cD_K$, namely the motivic exponential function $\sum_i g_i\cdot[X,0,h_i]$. This finishes the first step of the proofs.

In the second step, where we remove the residual parameterization, we proceed differently for the two theorems. For Theorem \ref{p2pexpmot}, let $h_{j,par}$ for $j=1,2,3$ be obtained using the first step, using a residual parameterization $\sigma$. These $h_{j,par}$  live in $\cCexp(X[0,s,0])$ where the $s$ residue field variables were introduced with $\sigma$.
We define the final $h_j$ by $\mu_X(h_j\overline h_j)$, where $\overline h_j$ is given by Lemma \ref{conjumot}, and where $\mu_X$ is as in (8.7.10) of \cite{CLexp} for the coordinate projection from $X[0,s,0]$ to $X$ (this is integration on the fibers which in our case means integration over the
$s$ residue variables). Similarly to the argument ending the proof of Theorem \ref{p2mot}, these $h_j$ are as desired by Theorem \ref{p2pexpmot} for $f$. For Theorem \ref{p2'pexpmot}, let $g_{par}$ be obtained using the first step with residual parameterization $\sigma$.
One defines $g$ as $\sigma^*(g_{par})$ and concludes as in the proof of Theorem \ref{p2'mot}.
\end{proof}

\begin{proof}[Proof of Theorems \ref{p1qmot} and \ref{p1qmotexp}]
The Specialization Principle \ref{SpecialP} yields that taking motivic integrals, for integrable functions,  commutes with specialization for $K$ and $\psi$ for all $K\in \CO[,M]$ and $\psi$ in $\cD_K$, for large enough $M$. Now the results follow from Theorems \ref{p2'mot} and \ref{p2'pexpmot}, in exactly the same way as Theorems \ref{p1p-new} and \ref{p1pexp} follow from Theorems \ref{p2'p} and \ref{p2'pexp}.
\end{proof}

Finally, we can can give the proof of our new transfer principles, which follows from the work we have done.

\begin{proof}[Proof of Theorems \ref{mtrel} and \ref{mtbound}]
For the first statement of Theorem \ref{mtrel}, resp.~of Theorem \ref{mtbound}, take $h_1$, resp. $h_2$, as given by Theorem \ref{p2pexpmot}. For the second statement of Theorem \ref{mtrel}, resp.~of Theorem \ref{mtbound}, take $h_1$, resp. $h_2$, as given by Corollary \ref{p2pmotlocexp}. In all cases the proof is finished by applying Proposition \ref{strong} to $h_1$, resp.~to $h_2$.
\end{proof}

Note that the basic form of the transfer principle for integrability as stated in the introduction follows from Theorem \ref{mtrel}. Indeed, in that basic form, the functions $F_{K,\psi}$ for $K$ in $\CO$ are a special case of functions that come from a motivic exponential function since $\LPas$ is richer than the language of valued fields that is used in the introduction. Moreover, the family of additive characters $x\mapsto \psi(y x)$ with nonzero parameter $y$ in the valued field and for $\psi\in\cD_K$ clearly allows one to deduce the basic form (which does not impose conditions on the conductor) from the results we have established. First use Theorem \ref{p2pexpmot}, for the family with nonzero $y$ as parameter in the character as we just described, to find $h_1$ in $\cCexp(h[1,0,0]$, extended by zero on zero. Use Theorem \ref{p2pexpmot} for this $h_1$ to find a new function $h_3$ in $\cCexp(h[0,0,0])$. Finally  apply Proposition \ref{strong} to $h_3$ to recover the form of the transfer principle as
stated in the introduction.

\bibliographystyle{amsplain}
\bibliography{tibib}

\end{document}